\documentclass[12pt]{article}
\usepackage[centertags]{amsmath}
\usepackage{amsfonts}
\usepackage{amssymb}
\usepackage{amsthm}
\usepackage{amsmath,mathrsfs}
\usepackage{amsmath,amscd}
\usepackage[matrix,arrow,curve,cmtip]{xy}
\title{\textbf{Supersymmetric Derived Stacks}}
\author{Renaud Gauthier \footnote{rg.mathematics@gmail.com} \\ \\}
\theoremstyle{definition}

\newtheorem{ksModstSymmMonCat}{Theorem}[subsubsection]
\newtheorem{Xilim}{Lemma}[subsubsection]
\newtheorem{SLadj}[Xilim]{Proposition}
\newtheorem{fibr}[Xilim]{Proposition}
\newtheorem{Eta}[Xilim]{Proposition}
\newtheorem{SModStr}[Xilim]{Theorem}
\newtheorem{sCdMod}{Theorem}[subsubsection]
\newtheorem{sksModModcat}[sCdMod]{Corollary}
\newtheorem{grscat}{Definition}[subsubsection]
\newtheorem{sksModgrscat}[grscat]{Theorem}
\newtheorem{grSimpModCat}{Graded simplicial model category Axiom - grsModCat}[subsubsection]
\newtheorem{grsModCat}[grSimpModCat]{Definition}
\newtheorem{PropgrsModCat}[grSimpModCat]{Proposition}
\newtheorem{sksModgrsModCat}[grSimpModCat]{Theorem}
\newtheorem{Ghom}[grSimpModCat]{Lemma}
\newtheorem{sksModstSymmMonModCat}{Theorem}[subsubsection]
\newtheorem{der}{Definition}[subsection]
\newtheorem{MapDer}[der]{Proposition}
\newtheorem{LBAdef}[der]{Definition}
\newtheorem{formetale}{Definition}[subsection]
\newtheorem{etale}[formetale]{Definition}
\newtheorem{htpycart}{Definition}[subsection]
\newtheorem{resdiagr}{Definition}[subsection]
\newtheorem{ass1}{Assumption 1}[subsection]
\newtheorem{ass2}{Assumption 2}[subsection]
\newtheorem{ass3}{Assumption 3}[subsection]
\newtheorem{ass4}{Assumption 4}[subsection]

\DeclareMathOperator*{\colim}{\text{colim}}

\DeclareMathOperator*{\surjfleche}{\twoheadrightarrow}
\newcommand{\beq}{\begin{equation}}
\newcommand{\eeq}{\end{equation}}

\newcommand{\rarr}{\rightarrow}

\newcommand{\Ob}{\text{Ob\,}}
\newcommand{\surj}{\twoheadrightarrow}
\newcommand{\xrarr}{\xrightarrow}

\newcommand{\cA}{\mathcal{A}}

\newcommand{\cC}{\mathcal{C}}

\newcommand{\cH}{\mathcal{H}}

\newcommand{\cX}{\mathcal{X}}

\newcommand{\bL}{\mathbb{L}}

\newcommand{\bR}{\mathbb{R}}

\newcommand{\bZ}{\mathbb{Z}}

\newcommand{\Hom}{\text{Hom}}
\newcommand{\Ho}{\text{Ho}\,}

\newcommand{\Mor}{\text{Mor}\,}
\newcommand{\Map}{\text{Map}}

\newcommand{\op}{\text{op}}

\newcommand{\Set}{\text{Set}}
\newcommand{\Spec}{\text{Spec\,}}
\newcommand{\Top}{\text{Top}}
\newcommand{\uHom}{\underline{\Hom}}

\newcommand{\ACommC}{A\text{-Comm}(\cC)}

\newcommand{\AMod}{A\text{-Mod}}
\newcommand{\AsMod}{A\text{-sMod}}

\newcommand{\AffuC}{\text{Aff}_{\ucC}}

\newcommand{\AffC}{\cA \text{ff}_{\cC}}

\newcommand{\bTA}{\boxtimes_A}
\newcommand{\bT}{\boxtimes}

\newcommand{\BCommC}{B\text{-Comm}(\cC)}

\newcommand{\BsMod}{B\text{-sMod}}

\newcommand{\cDot}{\centerdot}

\newcommand{\cCstdot}{\cC_{*\cDot}}
\newcommand{\cst}{\text{cst}}
\newcommand{\cHom}{\cH \text{om}}
\newcommand{\covet}{\text{cov}_{\text{\'et}}}
\newcommand{\csAstsMod}{\text{cs}A_*\text{-sMod}}

\newcommand{\csBstMod}{\text{cs}B_*\dashMod}
\newcommand{\csBstsMod}{\text{cs}B_*\text{-sMod}}
\newcommand{\csAsMod}{\text{cs}A\text{-sMod}}

\newcommand{\Comm}{\text{Comm}}

\newcommand{\CommC}{Comm(\cC)}

\newcommand{\Dn}{\Delta^n}

\newcommand{\dashMod}{\text{-Mod}}
\newcommand{\dashsMod}{\text{-sMod}}

\newcommand{\Deltastdot}{\Delta_{*\centerdot}}
\newcommand{\dksAff}{\text{d}k\text{-sAff}}
\newcommand{\dksAfftet}{\dksAff^{\;\sim \, , \, \acute{e}t.}}
\newcommand{\Der}{\mathbb{D}\text{er}}
\newcommand{\DerA}{\mathbb{D}\text{er}_A}

\newcommand{\Dop}{\Delta^{\op}}

\newcommand{\eps}{\epsilon}

\newcommand{\fstdot}{f_{*\cDot}}

\newcommand{\Gstdot}{G_{*\cDot}}
\newcommand{\gstdot}{g_{*\cDot}}

\newcommand{\holim}{\text{holim}}

\newcommand{\hatdksAff}{\widehat{\dksAff}}
\newcommand{\hocolim}{\text{hocolim}\,}
\newcommand{\HomSd}{\Hom_{\Sd}}
\newcommand{\HomSstdot}{\Hom_{\Sstdot}}
\newcommand{\HomS}{\Hom_S}

\newcommand{\HomsC}{\Hom_{s\cC}}
\newcommand{\HomuC}{\Hom_{\underline{\cC}}}
\newcommand{\HomsuC}{\Hom_{s\underline{\cC}}}

\newcommand{\hstdot}{h_{*\cDot}}
\newcommand{\Hd}{H_{\cDot}}

\newcommand{\kMod}{k\text{-Mod}}

\newcommand{\kDAff}{\text{k-D}^{-}\text{Aff}}

\newcommand{\kDsAff}{k\text{-D}^-\text{sAff}}
\newcommand{\kDsAfftet}{\kDsAff^{\sim \, , \, \acute{e}t.}}
\newcommand{\ksMod}{k\text{-sMod}}

\newcommand{\ksAlg}{k\text{-sAlg}}

\newcommand{\LBA}{\mathbb{L}_{B/A}}
\newcommand{\LA}{\mathbb{L}_A}

\newcommand{\LB}{\mathbb{L}_B}

\newcommand{\LBipAip}{\mathbb{L}_{B_{ip}/A_{ip}}}

\newcommand{\LCM}{\text{L}_{\cC}M}

\newcommand{\Mhat}{M^{\wedge}}

\newcommand{\Mop}{M^{\op}}

\newcommand{\Mtet}{M^{\sim \,, \,\tet}}

\newcommand{\MapC}{\Map_{\cC}}

\newcommand{\oP}{\oplus}
\newcommand{\oPi}{\oP_{i=0,1}}
\newcommand{\oT}{\otimes}
\newcommand{\oTA}{\oT_A}
\newcommand{\oTAL}{\oTA^{\mathbb{L}}}

\newcommand{\PrZtwo}{\text{Pr}_{\mathbb{Z}_2}}
\newcommand{\pizeropr}{\pi_0^{\text{pr}}}
\newcommand{\pizeroeq}{\pi_0^{\text{eq}}}
\newcommand{\pDn}{\partial \Delta^n}

\newcommand{\RomuHom}{\bR_{\omega}\uHom}

\newcommand{\RpDn}{\mathbb{R}\pDn}
\newcommand{\RDn}{\mathbb{R}\Dn}
\newcommand{\Ruh}{\mathbb{R}\uh}
\newcommand{\Rjxstar}{\mathbb{R}j_x^*}

\newcommand{\Rjxipstar}{\bR j_{x_{ip}}^*}
\newcommand{\RuHom}{\bR \uHom}

\newcommand{\sPr}{\text{sPr}}

\newcommand{\SetD}{\Set_{\Delta}}

\newcommand{\skMod}{\text{s}k\text{-Mod}}
\newcommand{\skCAlg}{\text{s}k\text{-CAlg}}

\newcommand{\sAff}{s\text{Aff}}

\newcommand{\sPrZtwosqu}{\sPr_{\bZ^2_2}}
\newcommand{\Sd}{S_{\cDot}}
\newcommand{\Sstdot}{S_{* \centerdot}}

\newcommand{\sksMod}{\text{s}k\text{-sMod}}
\newcommand{\sAffC}{\sAff_{\cC}}
\newcommand{\sAffuC}{\sAff_{\ucC}}
\newcommand{\sksAlg}{\text{s}k\text{-sAlg}}

\newcommand{\scCstdot}{s\cCstdot}

\newcommand{\tet}{\acute{e}t.}
\newcommand{\Topstdot}{\Top_{* \centerdot}}
\newcommand{\tbT}{\tilde{\boxtimes}}
\newcommand{\tbTA}{\tbT_A}
\newcommand{\tbTAL}{\tbTA^{\mathbb{L}}}

\newcommand{\uh}{\underline{h}}

\newcommand{\uMhat}{\uM^{\wedge}}
\newcommand{\uHomC}{\underline{\Hom}_{\cC}}
\newcommand{\uhomC}{\underline{\hom}_{\cC}}
\newcommand{\uhom}{\underline{\hom}}
\newcommand{\uhomuC}{\uhom_{\ucC}}
\newcommand{\uHomsC}{\uHom_{s\cC}}

\newcommand{\uhomSstdot}{\uhom_{\Sstdot}}
\newcommand{\ucHom}{\underline{\cHom}}

\newcommand{\uM}{\underline{M}}
\newcommand{\uHomA}{\underline{\Hom}_A}

\newcommand{\ucC}{\underline{\cC}}

\newcommand{\Xstdot}{X_{* \cDot}}
\newcommand{\xstdot}{x_{*\cDot}}

\newcommand{\Ystdot}{Y_{* \cDot}}

\begin{document}
\maketitle
\begin{abstract}
Stacks have become a prevalent tool in studying problems with connections to String Theory, hence we see a need to develop a theory of supersymmetric stacks proper. We first define derived stacks on $\mathbb{Z}_2$-bi-graded $k$-modules (objects of $\sksMod_*$) following the exposition of Toen and Vezzosi on ungraded modules in \cite{TV} and \cite{TV4}. We then define $\Topstdot$-valued maps on those supermodules ($\Topstdot$ $\bZ_2$-bi-graded) and show how they behave under supersymmetry transformations in the base. For $\Psi: M \rarr \cX$ one such map, $M \in \sksMod_*$, $\cX \in \Topstdot$, we argue that defining $F(M) = \{\Psi(\sigma, \theta) \;|\; (\sigma, \theta) \in M \}$ with the induced topology, one can call $F$ a supersymmetric stack if it is a derived stack.
\end{abstract}

\newpage

\section{Introduction}
Moduli spaces have for a long time been a pervasive object in String Theory. Moduli stacks in particular have been very useful, and with the advent of Derived Algebraic Geometry (\cite{Lu}, \cite{TV}, \cite{TV4}), derived moduli stacks. Geometric Langlands is another avenue of research that has deep connections with Mathematical Physics (\cite{AT}, \cite{EY}), and makes heavy use of the stacks formalism. Naturally then it seems appropriate to develop a theory of supersymmetric stacks in its own right since that notion ought to surface at some point.\\

One typically starts with a superspace, a graded vector space. To be more general, one can consider graded modules instead, over a commutative ring $k$, giving rise to a category of super-$k$-modules $\ksMod$. We work in Derived Algebraic Geometry, so we enlarge that to the category of simplicial graded $k$-supermodules $\sksMod_*$. From that point onward we follow \cite{TV} and \cite{TV4} by showing $\sksMod_*$ can be made into a symmetric monoidal model category, with a notion of supercommutativity using the parity function on homogeneous elements that leads to defining $\sksAlg_* = \Comm(\sksMod_*)$ the category of simplicial graded $k$-superalgebras, the opposite category of which is the category of affine graded $k$-superschemes $\kDsAff_*$, on which we put a bi-graded version of the \'{e}tale topology of \cite{TV4}. We consider bi-graded stacks on this site, functors valued in $\Sstdot = (\SetD)_{2 \times 2}$, a bi-graded generalization of the notion of stacks on the ungraded site $(\kDAff, \tet)$. The passage from the classical theory as expounded in \cite{TV} and \cite{TV4} to the bi-graded setting is made possible precisely because of our choice of working with bi-graded objects, along with our choice of tensor product on $\sksMod_*$.\\

For each $M \in \sksMod_*$, $M = (M_{ip})$, $i$ referred to as the level index, $p$ the parity index, $\sigma = (\sigma^1, \sigma^2)$, $\sigma^i \in M_{i0}$, $\theta = (\theta^1, \theta^2)$, $\theta^A \in M_{A1}$, we consider maps $\Psi: M \rarr \cX$, $\cX = (\cX_{ip}) \in \Topstdot = (\Top)_{2 \times 2}$, $\cX$ determined by additional data such as constraint equations, $\Psi = X \oP \psi$, $X: M_0 \rarr \cX_0$, $\psi: M_1 \rarr \cX_1$, so that $X$ is a function of $\sigma = (\sigma^1, \sigma^2)$, and $\psi$ is a function of $\theta = (\theta^1, \theta^2)$. Instead of using the traditional supersymetry transformations, say those in \cite{GSW} for which $M_0 = < \sigma^1, \sigma^2>$, $M_1 = < \theta^1, \theta^2>$,  $\delta \sigma^{\alpha} = \eps^T \rho^{\alpha} \theta$, with $\eps$ anticommuting, $\delta \theta^A = \eps^A$, we introduce more general transformations that are fully symmetric, both algebraically, but also parity-wise, i.e. transformations $\delta = \delta_0 \oP \delta_1$, $\delta_0: M_0 \rarr M_1$ and $\delta_1: M_1 \rarr M_0$ such that $\delta_0 \sigma^{\alpha} = \eps^T \rho^{\alpha} \theta$, and $\delta_1 \theta^A = \lambda^T \gamma^A \sigma$, for infinitesimal commuting parameters $\eps$ and $\lambda$, and matrices $\rho$ and $\gamma$. We then show that $\Psi$ transforms via a pullback under $\delta$. \\

For $A \in \sksAlg_*$, define $F(A) = \{\Psi(\sigma, \theta) \, | \, \sigma \in A_0, \theta \in A_1 \}$ endowed with the subspace topology. Those functors $F$ that satisfy hyperdescent on $(\kDsAff_*, \tet)$ are called supersymmetric (derived) stacks.

\section{Simplicial Super algebra}
In a first time we remind the reader of various results pertaining to super-algebra that we will later generalize to the simplicial setting. Most of what is presented in the subsection below can be found in \cite{V}, and especially in \cite{CCF}, which we modify for our purposes.

\subsection{Super Algebra}
This subsection covers some basic conventions about superalgebra. We fix $k$ a commutative ring. We define a super $k$-module (also referred to as $k$-supermodule) to a be a $\mathbb{Z}_2$-graded $k$-module $M = M_0 \oplus M_1$, endowed with a parity function $|\;|$ defined by:
\beq
|x|=
\left\{
  \begin{array}{ll}
    0, & \hbox{$x \in M_0$} \\
    1, & \hbox{$x \in M_1$}
  \end{array}
\right. \nonumber
\eeq
There is a (distributive) tensor product $\otimes$ on the set of $k$-supermodules $\ksMod$, for which the commutativity map $\sigma$ (also referred to as braiding map), defined on objects by $\sigma_{M,N}: M \otimes N \rarr N \otimes M$, $M = M_0 \oplus M_1$, $N=N_0 \oplus N_1$, satisfies:
\beq
\sigma_{M,N}(x \otimes y) = (-1)^{|x||y|} y \otimes x \nonumber
\eeq
on homogeneous elements, and is extended by linearity. Morphisms of super modules are defined to be graded morphisms, or equivalently morphisms of graded objects of degree zero (\cite{McL0}). In other terms if $f: M \rarr N$ is a morphism of super $k$-modules, then it decomposes into two morphisms $f_0: M_0 \rarr N_0$ and $f_1: M_1 \rarr N_1$. We denote by $\Hom_k(M, N)$, or simply by $\Hom(M,N)$ the set of graded morphisms from $M$ to $N$. It is a $k$-module. Each super module $M$ has an identity $id_M = id_{M_0} \oplus id_{M_1} \in \Hom(M,M)$ that acts as a left and a right identity for composition. With obvious notations, composition is defined by $f \circ g = f_0 \circ g_0 \oP f_1 \circ g_1$, and is associative. This makes $\ksMod$ into a category. \\

The identity for the tensor product is $k$, and we have left and right identity maps:
\begin{align}
\lambda:\; & k \otimes x \xrightarrow{\cong} x \nonumber \\
\rho: \;&x \otimes k \xrightarrow{ \cong} x \nonumber
\end{align}
with $x \in M$, $M$ an object of $\ksMod$. Here for instance $k\otimes (x_0 \oplus x_1) = (k \otimes x_0) \oplus (k \otimes x_1)$, so that we can write $\lambda = \lambda_0 \oP \lambda_1$ and $\rho = \rho_0 \oP \rho_1$, $\lambda_i$ and $\rho_i$ the classical left and right identity maps in $\kMod$ respectively. We regard elements of $k$ to be of degree zero so that they preserve parity with the tensor product. The associator is defined by:
\beq
\alpha_{M,N,P}^{-1}: M \otimes (N \otimes P) \xrightarrow{\cong} (M \otimes N) \otimes P \nonumber
\eeq
and satisfies the pentagon coherence diagram (\cite{McL}):
\beq
\xymatrix{
&(M \otimes N) \otimes (P \otimes Q) \ar[dr]^{\alpha} \\
((M \otimes N) \otimes P) \otimes Q \ar[ur]^{\alpha} \ar[d]_{\alpha \oT Q} &&M \otimes (N \otimes (P \otimes Q)) \\
(M \otimes (N \otimes P)) \otimes Q \ar[rr]_{\alpha} &&M \otimes ((N \otimes P) \otimes Q) \ar[u]_{M \oT \alpha} } \nonumber
\eeq
\noindent
as well as the triangle coherence diagrams:
\beq
\xymatrix{
(M\otimes k) \otimes N \ar[dr]_{\rho \otimes N} \ar[rr]^{\alpha} &&M \otimes ( k \otimes N) \ar[dl]^{M \otimes \lambda}\\
&M \otimes N } \nonumber
\eeq
\noindent
and the bigon relations $\rho_k = \lambda_k$. \\

The braiding map $\sigma$ is required to satisfy the hexagon coherence condition:
\beq
\xymatrix{
&M \otimes (N \otimes P) \ar[r]^{\sigma^{\pm 1}} &(N \otimes P) \otimes M \ar[dr]^{\alpha} \\
(M \otimes N) \otimes P \ar[ur]^{\alpha} \ar[dr]_{\sigma^{\pm 1} \otimes P} &&&N \otimes (P \otimes M) \\
&(N \otimes M)\otimes P \ar[r]_{\alpha} &N \otimes (M \otimes P) \ar[ur]_{N \otimes \sigma^{\pm 1}} } \nonumber
\eeq
We investigate what does the parity function have to satisfy for this to hold: we start with an element $(x \otimes y) \otimes z$ in $(M \otimes N) \otimes P$, with $x = x_0 \oP x_1$, $y = y_0 \oP y_1$ and $z = z_0 \oP z_1$. Thus, abbreviating tensor products $x \oT y$ by $xy$ and direct sums $x \oP y$ by $x+y$ for simplicity of notation:
\begin{align}
(x \oT y) \oT z & = [(x_0 \oP x_1) \oT (y_0 \oP y_1)] \oT  z \nonumber \\
&=[x_0y_0 + x_0 y_1 + x_1 y_0 + x_1y_1] \oT z \nonumber
\end{align}
maps to:
\beq
[y_0x_0 + y_1 x_0 + y_0x_1 - y_1 x_1] \oT z \nonumber
\eeq
under $\sigma \otimes z$, which then maps to:
\beq
y_0(x_0 \otimes z) + y_1(x_0 \otimes z) + y_0(x_1 \oT z) - y_1(x_1 \oT z) \nonumber
\eeq
under $\alpha$. This expands as:
\beq
y_0(x_0z_0 + x_0z_1) + y_1(x_0z_0+x_0z_1) + y_0(x_1z_0 + x_1z_1) - y_1(x_1z_0 + x_1z_1) \nonumber
\eeq
and this maps under $1 \oT \sigma$ to:
\beq
y_0(z_0x_0+z_1x_0) + y_1(z_0x_0 + z_1x_0) + y_0(z_0x_1-z_1x_1) - y_1(z_0x_1-z_1x_1) \label{Rder}
\eeq
Now starting from the same object $(x \oT y) \oT z$, under $\alpha$ this maps to:
\beq
(x_0+x_1) \oT [y_0z_0 + y_0z_1 + y_1z_0 + y_1z_1] \nonumber
\eeq
which expands as:
\beq
x_0 \oT y_0z_0 + x_0 \oT y_0z_1 + x_0 \oT y_1z_0 + x_0 \oT y_1z_1 + x_1 \oT y_0z_0 + x_1 \oT y_0z_1 + x_1 \oT y_1z_0 + x_1 \oT y_1z_1 \nonumber
\eeq
Now if we have:
\beq
|x \oT y| = |x| + |y| \label{xty}
\eeq
then this expansion maps to:
\beq
y_0z_0 \oT x_0 + y_0z_1 \oT x_0 + y_1z_0 \oT x_0 + y_1z_1 \oT x_0 + y_0z_0 \oT x_1 - y_0z_1 \oT x_1 - y_1z_0 \oT x_1 + y_1z_1 \oT x_1 \nonumber
\eeq
under $\sigma$, which itself maps to:
\beq
y_0(z_0x_0) + y_0(z_1x_0) + y_1(z_0x_0) + y_1(z_1x_0) + y_0(z_0x_1) - y_0(z_1x_1) - y_1(z_0x_1) + y_1(z_1x_1) \nonumber
\eeq
under $\alpha$, and this is exactly \eqref{Rder}, showing that the hexagonal diagram does commute with the definition of the parity on tensor products given by \eqref{xty}. To see that \eqref{xty} is indeed the correct relation, it suffices to circle the hexagon diagram both ways by starting from a same element $(x \oT y) \oT z$, $x$, $y$, $z$ homogeneous elements, as in the diagram below:
\beq
\xymatrix{
x \oT ( y \oT z) \ar[r]^{\sigma^{\pm 1}} &(-1)^{|y \oT z||x|}(y \oT z) \oT x \ar[dr]^{\alpha}\\
(x \oT y) \oT z \ar[u]^{\alpha} \ar[dd]_{\sigma^{\pm 1} \otimes z} &&(-1)^{|y \oT z||x|} y \oT (z \oT x) \ar@{=}[d] \\
&&(-1)^{|y||x|} (-1)^{|x||z|}y \oT (z \oT x)\\
(-1)^{|y||x|}(y \oT x) \oT z \ar[r]_{\alpha} &(-1)^{|y||x|}y \oT (x \oT z) \ar[ur]_{y \otimes \sigma^{\pm 1}}}\nonumber
\eeq
It follows from \eqref{xty} that if the parity function is calculated mod 2, the tensor product on $\ksMod$ is defined by:
\begin{align}
(M \otimes N)_0 &= (M_0 \otimes N_0) \oplus (M_1 \otimes N_1) \nonumber \\
(M \otimes N)_1 &= (M_0 \otimes N_1) \oplus (M_1 \otimes N_0) \nonumber
\end{align}
where $M_i \otimes N_j$, $i,j=0,1$ is the usual tensor product in $\kMod$ on the underlying $k$-modules. Finally, the braiding $\sigma = \sigma_0 \oP \sigma_1$ on $\ksMod$ satisfies $\sigma^2 = \sigma_0^2 \oP \sigma_1^2 = 1$. Indeed:
\begin{align}
	M \otimes N \ni x \otimes y &= x_0y_0 + x_1y_1 \xrarr{\sigma_0} y_0x_0 - y_1 x_1 \xrarr{\sigma_0} x_0y_0 -(-x_1y_1) \in (N \otimes M)_0 \nonumber \\
	& + x_0y_1 + x_1y_0 \xrarr{\sigma_1} +y_1x_0 +y_0x_1 \xrarr{\sigma_1} x_0y_1 + x_1y_0 \in (N \otimes M)_1 \nonumber
\end{align}
where we have omitted tensor products for ease of reading. Further:
\begin{align}
\rho_M &= \rho_{M,0} \oP \rho_{M,1} \nonumber \\
	&= \oPi \lambda_{M,i} \circ (\sigma_{M,k})_i \nonumber \\
 &= \lambda_M \circ \sigma_{M,k} \nonumber
\end{align}
meaning, taking $k$ to have even parity, the following diagram commutes:
\beq
\xymatrix{
M \oT k \ar[dr]_{\rho_M} \ar[rr]^{\sigma_{M,k}} && k \oT M  \ar[dl]^{\lambda_M} \nonumber \\
&M \nonumber
}
\eeq
At this point we have shown that $(\ksMod, \oT_k, \alpha, \lambda, \rho, k)$ is a symmetric monoidal category.\\

\subsection{Diagonal super algebra}
\subsubsection{Box tensor product $\bT$}
In order to develop a theory within an Algebraic Geometry context over graded $k$-modules, we need commutative monoids whose definition really makes sense only if we take a diagonal version of $(\ksMod, \oT_k, \alpha, \lambda, \rho, k)$. One way to achieve this is by considering $\ksMod_*$, the category of graded $k$-supermodules. We take the grading to be $\bZ_2$ as well. An object of $\ksMod_*$ is of the form $M = M_0 \oplus M_1$, $M_i = M_{i0} \oplus M_{i1}$ for $i=0,1$, where objects of $M_{ip}$ have parity $p$. Thus one can write $M = \oplus_{i,p = 0,1}M_{ip}$, $i$ referring to the graded index, or level, and $p$ referring to parity. We will adopt the shorthand matricial notation $M = (M_{ip})$ to denote such objects. It follows that we will refer to the components $M_{ip}$ as entries of the object $M \in \ksMod_*$. In the same manner, morphisms $f: M \rarr N$ in $\ksMod_*$ are defined entrywise: $f_{ip}: M_{ip} \rarr N_{ip}$. The tensor product on $\ksMod_*$ is a diagonal tensor product defined by:
\beq
M \bT N = (M_0 \otimes N_0) \oplus (M_1 \otimes N_1) \nonumber
\eeq
where the tensor products used on the right hand side are those of $\ksMod$. Observe that this is performed levelwise and each level is closed under $\otimes$: $M_{ip} \otimes N_{iq} \in \ksMod_i$, $i=0,1$. It follows $M_i \otimes N_i = (M \bT N)_i$. The advantage of having such a definition of the tensor product on $\ksMod_*$ is that one can use classical module theory with elements that have a parity. We have a braiding:
\beq
\sigma_{M,N}: M \bT N \rarr N \bT M \nonumber
\eeq
defined levelwise on supermodules:
\beq
\sigma_{M,N}: (M_0 \otimes N_0) \oplus (M_1 \otimes N_1) \xrarr{\sigma_0 \oplus \sigma_1} (N_0 \otimes M_0) \oplus (N_1 \otimes M_1) = N \bT M \nonumber
\eeq
with $\sigma_i = \sigma_{i0} \oplus \sigma_{i1}: M_i \otimes N_i \rarr N_i \otimes M_i$. To be more precise:
\begin{align}
	\sigma_{i0}:& (M_{i0} \otimes N_{i0}) \oplus (M_{i1} \otimes N_{i1}) \rarr (N_{i0} \otimes M_{i0}) \oplus (N_{i1} \otimes M_{i1}) \nonumber \\
	\sigma_{i1}:& (M_{i0} \otimes N_{i1}) \oplus (M_{i1} \otimes N_{i0}) \rarr (N_{i1} \otimes M_{i0}) \oplus (N_{i0} \otimes M_{i1}) \nonumber
\end{align}

\subsubsection{$\ksMod_*$ symmetric monoidal category}

Recall that classically morphisms of supermodules $M \rarr N$ are graded morphisms $f = f_0 \oP f_1$, $f_0: M_0 \rarr N_0$, $f_1: M_1 \rarr N_1$. In the bi-graded case with the box tensor product defined above, morphisms in $\ksMod_*$ are defined levelwise, and on each level, morphisms are morphisms in $\ksMod$. Thus $f: M \rarr N$ in $\ksMod_*$ decomposes as $f = f_0 \oplus f_1$, with $f_i: M_i \rarr N_i$, $i=0,1$, and for $i$ fixed, $f_i = f_{i0} \oplus f_{i1}$, with $f_{ip}: M_{ip}\rarr N_{ip}$, with $p = 0,1$. Thus $f = \oplus_{i,p = 0,1} f_{ip}$, which we can represent matricially by $ f = (f_{ip})$. \\

We still denote by $\Hom(M,N)$ the set $\Hom_{\ksMod_*}(M,N)$. We have $\Hom_{\ksMod_*} = \Hom_{(\ksMod_0, \oT)} \oP \Hom_{(\ksMod_1, \oT)}$. It is a $k$-module. For any $M \in \ksMod_*$, $id_M = id_{M_0} \oP id_{M_1} \in \Hom(M,M)$. Composition is associative. This makes $\ksMod_*$ into a category. We have:
\begin{align}
	&\lambda = \lambda_0 \oP \lambda_1: \; \Delta k \bT x = (k \otimes x_0) \oplus (k \otimes x_1) \xrightarrow{\cong} x   \nonumber \\
&\rho = \rho_0 \oP \rho_1: \;  x \bT \Delta k \xrightarrow{\cong} x   \nonumber \\
&\alpha_{M,N,P}: \;  (M \bT N) \bT P \xrightarrow{\cong} M \bT ( N \bT P) \nonumber
\end{align}
To be more precise:
\begin{align}
\alpha_{M,N,P}: &(M \bT N) \bT P \nonumber \\
=& \oPi (M \bT N)_i \oT P_i \nonumber \\
=& \oPi (M_i \oT N_i) \oT P_i \xrightarrow{\oPi \alpha_i} \oPi M_i \oT (N_i \oT P_i) = M \bT (N \bT P) \nonumber
\end{align}
with $\alpha_0$ and $\alpha_1$ associators in $\ksMod$, so we do have:
\beq
\alpha = \alpha_0 \oP \alpha_1: (M \bT N ) \bT P \xrightarrow{\cong} M \bT ( N \bT P) \nonumber
\eeq
We have a pentagon diagram:
\beq
\xymatrix{
(M \bT N) \bT (P \bT Q) \ar@{=}[d] \\
\oPi (M_i \oT N_i) \oT (P_i \oT Q_i) \ar[rdd]^{\oPi \alpha_i} \\
\oPi ((M_i \oT N_i) \oT P_i) \oT Q_i \ar[u]^{\oPi \alpha_i} \ar@{=}[d] &M \bT ( N \bT (P \bT Q)) \ar@{=}[d]\\
((M \bT N) \bT P) \bT Q \ar[d]_{\oPi \alpha_i \oT Q_i} & \oPi M_i \oT(N_i \oT(P_i \oT Q_i)) \\
\oPi (M_i \oT(N_i \oT P_i)) \oT Q_i \ar@{=}[d] \ar[r]_{\oPi \alpha_i} & \oPi M_i \oT ((N_i \oT P_i) \oT Q_i) \ar[u]_{\oPi M_i \oT \alpha_i} \ar@{=}\\
(M \bT (N \bT P)) \bT Q & M \bT((N \bT P) \bT Q) \ar@{=}[u] \nonumber
}
\eeq
and triangle coherence diagrams:
\beq
\xymatrix{
(M \bT \Delta k ) \bT N \ar@{=}[d]  &&M \bT (\Delta k \bT N) \ar@{=}[d] \\
\oPi (M_i \oT k) \oT N_i \ar[rr]^{\oPi \alpha_i} \ar[rd]_{\oPi \rho_i \oT N_i} &&\oPi M_i \oT (k \oT N_i) \ar[ld]^{\oPi M_i \oT \lambda_i} \\
&\oPi M_i \oT N_i \ar@{=}[d] \\
& M \bT N \nonumber
}
\eeq
with bigons $\rho_{\Delta k} = \lambda_{\Delta k}$. For the hexagonal coherence condition:
\beq
\xymatrix{
 M \bT (N \bT P) \ar@{=}[d] \ar@{-->}[r] & (N \bT P) \bT M \ar@{=}[d] \\
\oPi M_i \oT(N_i \oT P_i) \ar[r]^{\oPi \sigma_i} & \oPi (N_i \oT P_i) \oT M_i \ar[d]^{\oPi \alpha_i} \\
(M \bT N ) \bT P = \oPi (M_i \oT N_i) \oT P_i  \ar@{-->}@/_5pc/[dd] \ar[u]^{\oPi \alpha_i} \ar[d]^{\oPi \sigma_i \oT P_i} & \oPi N_i \oT( P_i \oT M_i) = N \bT (P \bT M)  \\
\oPi (N_i \oT M_i) \oT P_i \ar@{=}[d] \ar[r]^{\oPi \alpha_i} & \oPi N_i \oT(M_i \oT P_i) \ar@{=}[d] \ar[u]^{\oPi N_i \oT \sigma_i} \\
 (N \bT M) \bT P \ar@{-->}[r] & N \bT (M \bT P) \ar@{-->}@/_5pc/[uu] \nonumber
}
\eeq
with $\sigma^2 = 1$ and $\lambda_M = \rho_M \circ \sigma_{M, \Delta k}$ for symmetry. We have proved:
\begin{ksModstSymmMonCat}
	$(\ksMod_*, \bT, \alpha, \lambda, \rho, \Delta k)$ is a symmetric monoidal category.
\end{ksModstSymmMonCat}

\subsubsection{Commutative monoids in $\ksMod_*$}
A monoid in $\cC = \ksMod_*$ is an object $M$ of $\cC$ with an associative binary operation $\mu: M \bT M \rarr M$, $\mu = \mu_0 \oP \mu_1$, $\mu_i$ the ordinary binary operation in $\ksMod_i$ for $i = 0,1$. To be more precise:
\begin{align}
	\mu_i: M_i \otimes M_i &\rarr M_i \nonumber \\
	m_i \otimes m_i' & \mapsto m_im_i' \nonumber
\end{align}
where $\mu_i$ and $m_i \otimes m_i'$ decompose as:
\beq
(m_{i0} \otimes m_{i0}') \oplus (m_{i1} \otimes m_{i1}') \xrarr{\mu_{i0}} m_{i0} m_{i0}' \oplus m_{i1} m_{i1}' \nonumber 
\eeq
and:
\beq
(m_{i0} \otimes m_{i1}') \oplus (m_{i1} \otimes m_{i0}') \xrarr{\mu_{i1}} m_{i0} m_{i1}' \oplus m_{i1} m_{i0}' \nonumber 
\eeq
since each $\mu_{ip}$ is defined on homogeneous elements. The terms on the right hand side recombine as $(m_im_i')_0 \oplus (m_im_i')_1 = m_i m_i'$. \\

We also have a unit map $ \eta: \Delta k \rarr M$, $\eta = \eta_0 \oP \eta_1$, $\eta_i$ the unit in $\ksMod_i$ for $i = 0,1$. $\mu$ and $\eta$ satisfy, as in \cite{McL}, the coherence diagrams:
\beq
\xymatrix{
(M \bT M) \bT M \ar@{=}[d] & M \bT (M \bT M) \ar@{=}[d] & M \bT M \ar@{=}[d]\\
\oPi (M_i \oT M_i) \oT M_i \ar[d]_{\oPi \mu_i \oT M_i} \ar[r]^{\oPi \alpha_i} & \oPi M_i \oT (M_i \oT M_i) \ar[r]^-{\oPi M_i \oT \mu_i} & \oPi M_i \oT M_i \ar[d]^{\oPi \mu_i}\\
\oPi M_i \oT M_i = M \bT M\ar[rr]_{\oPi \mu_i} &&\oPi M_i = M } \nonumber
\eeq

and:
\beq
\xymatrixcolsep{5pc}
\xymatrix{
\Delta k \bT M \ar@{=}[d] \ar[r]^{\eta \bT M} & M \bT M \ar@{=}[d] & M \bT \Delta k \ar[l]_{M \bT \eta} \ar@{=}[d] \\
\oPi k \oT M_i  \ar[dr]_{\oPi \lambda_i} \ar[r]^{\oPi \eta_i \oT M_i} &\oPi M_i \oT M_i \ar[d]^{\oPi \mu_i} & \oPi M_i\oT k \ar[l]_{\oPi M_i \oT \eta_i} \ar[dl]^{\oPi \rho_i} \\
&\oPi M_i = M } \nonumber
\eeq
where $f \bT g = \oPi f_i \oT g_i$.\\

A monoid $A$ in $\ksMod_*$ is said to be super-commutative, which we will just refer to as being commutative, if $\mu \circ \sigma = \mu$, i.e. if levelwise, for all $x,y$ homogeneous in $A_i$ for $i = 0,1$, we have:
\beq
xy = \mu(x \oT y) = (-1)^{|x||y|} yx \nonumber
\eeq
Another way of saying this is $A \in \ksMod_*$ is supercommutative if it is so level wise.\\

We define the category of graded $k$-super algebras to be the category of commutative monoids (with commutativity defined by $\sigma$) in $\ksMod_*$:
\beq
\ksAlg_* = \Comm (\ksMod_*) \nonumber
\eeq
For $A \in \ksAlg_*$, we denote by $\AsMod_*$ the category of elements of $\ksMod_*$ that are $A$-modules in the classical sense, where the action of $A$ on $M \in \ksMod_*$ is defined levelwise. A morphism of $A$-modules is a morphism $f: M \rarr N$ of graded supermodules such that $f(am) = af(m)$ for $a \in A$, $m \in M$. More specifically, $am$ is defined via:
\begin{align}
\mu: A \bT M &\rarr M \nonumber \\
a \bT m & \mapsto am \nonumber
\end{align}
with $A \bT M = (A_0 \oT M_0) \oP (A_1 \oT M_1)$, $\mu = \mu_0 \oP \mu_1$ as we saw above, so that $am = a_0m_0 \oP a_1 m_1$, with $a_im_i \in M_i$, $i = 0,1$. Then $f(am) = af(m)$ reads:
\begin{align}
f(a_0m_0 \oP a_1 m_1) &= f_0(a_0m_0) \oP f_1(a_1m_1) \nonumber \\
&= a_0 f_0(m_0) \oP a_1 f_1(m_1) \nonumber \\
&=a_0f(m)_0 \oP a_1 f(m)_1 \nonumber \\
&=af(m) \nonumber
\end{align}
where $f_i \in \Mor(A_i \dashsMod_i)$ for $i = 0,1$. We denote by $\Hom_A(M,N)$ the set of morphism of $A$-modules. There is a tensor product on $\AsMod_*$ defined by:
\beq
M \bTA N = M \bT_k N / \sim \nonumber
\eeq
where the equivalence relation $\sim$ is defined by $ma \bT n = m \bT an$ for $a \in A$, $m \in M$, $n \in N$. However we will use a simplified product that will allow us to use the work of Toen and Vezzosi; we define:
\begin{align}
	X \tbTA Y = &\oPi [ \colim(X_{i0} \oT A_{i0} \oT Y_{i0} \rarr X_{i0} \oT Y_{i0}) \nonumber \\
	& \oP \colim( X_{i1} \oT A_{i1} \oT Y_{i1} \rarr X_{i0} \oT Y_{i1} \vee X_{i1} \oT Y_{i0})] \nonumber
\end{align}
that is:
\beq
X \tbTA Y = \oP_{i,p=0,1} X_{ip} \oT_{A_{ip}} Y_{ip} \nonumber
\eeq
which is fullly diagonalized.\\

\subsection{Simplicial Generalization}
We now define the simplicial counterparts to all the above definitions. We will not worry about universe considerations, and if needed they can be transcribed to our case from \cite{TV} and \cite{TV4}. We denote by $\sksMod_*$ the category of simplicial objects in $\ksMod_*$:
\beq
\sksMod_* = (\ksMod_*)^{\Delta^{\op}} \nonumber
\eeq
whose objects will be referred to as simplicial graded $k$-supermodules. We endow this category with the (simplicial) levelwise tensor product $\bT$, that is for $M,N \in \sksMod_*$, we have $M \bT N = \oP_{n \geq 0} M_n \bT N_n$, where:
\beq
M_n \bT N_n = \oP_{i= 0,1} M_{n,i} \oT_k N_{n,i} \nonumber
\eeq
We will put a model structure on this category. Following \cite{GoJa} it becomes evident that we will need a notion of model category structure on the category of bi-graded simplicial sets. We do this first.

\subsubsection{Model category of $\mathbb{Z}_2$-bi-graded simplicial sets}
In this subsection, we define a model category structure on the category of bi-graded simplicial sets, following \cite{GoJa}. For completeness' sake, we briefly remind the reader of a few results regarding simplicial sets, all of which can be found in \cite{GoJa}. Recall that a map $p: X \rarr Y$ of simplicial sets is said to be a fibration if for any commutative diagram:
\beq
\xymatrix{
\Lambda^n_k \ar@{^{(}->}[d]_i \ar[r] &X \ar[d]^p\\
\Dn \ar[r] \ar@{.>}[ru] &Y} \nonumber
\eeq
where as usual $\Dn = \Hom_{\Delta}(-,[n])$, there is a dotted map $ \Dn \rarr X$ making the above diagram commute. A fibrant simplicial set, or Kan complex, is a simplicial set $X$ for which $X \rarr *$ is a fibration, it being the unique map to the final object $*$. For $X$ a fibrant simplicial set, $v \in X_0$, we define the simplicial homotopy group $\pi_n(X,v)$ for $n \geq 1$ to be the set of homotopy classes of maps $\alpha: \Dn \rarr X$ rel $\partial \Dn$ that fit in diagrams of the form:
\beq
\xymatrix{
\Dn \ar[r]^\alpha &X \\
\partial \Dn \ar@{_{(}->}[u] \ar[r] &X_0 \ar[u]_v} \nonumber
\eeq
and define $\pi_0(X)$ to be the set of homotopy classes of vertices of $X$. A cofibration is an inclusion of simplicial sets. An equivalence is defined as follows: $f: X \rarr Y$ between fibrant simplicial sets is said to be a weak equivalence if for any vertex $x$ of $X$, $f_*: \pi_n(X,x) \rarr \pi_n(Y, f(x))$ is an isomorphism for all $n \geq 1$ and $f_*: \pi_0(X) \rarr \pi_0(Y)$ is a bijection.\\

We now define all those concepts in the bi-graded case. Let $\Topstdot$ denote the category of $\mathbb{Z}_2$-bi-graded topological spaces. Denote by $\Sstdot$ the category of $\mathbb{Z}_2$-bi-graded simplicial sets. We define the realization functor $|\;|: \Sstdot \rarr \Topstdot$ entrywise. Introduce the bi-graded simplex category $\Deltastdot \downarrow X_{*\centerdot}$ of a bi-graded simplicial set $X_{*\centerdot}$, whose objects are bi-graded maps which entry-wise read $\sigma_{ip}: \Dn \rarr X_{ip}$, $i,p=0,1$.\\

An arrow of $\Deltastdot \downarrow X_{*\centerdot}$ consists of commutative diagrams of simplicial maps for $i,p=0,1$:
\beq
\xymatrix{
	\Dn \ar[dr]^{\sigma_{ip}} \ar[dd]_{\theta_{ip}} \\
&X_{ip} \\
\Delta^m \ar[ur]_{\tau_{ip}}} \nonumber
\eeq
We have the following result from \cite{GoJa}:
\begin{Xilim}
	$X_{ip} \cong \colim_{\substack{\Dn \rarr X_{ip} \\ \text{ in } \Delta \downarrow X_{ip}}} \Dn \in \SetD$.
\end{Xilim}
Using the following definition on usual simplicial sets:
\beq
|X_{ip}| = \colim_{\substack{\Dn \rarr X_{ip} \\ \text{ in } \Delta \downarrow X_{ip}}} |\Dn| \nonumber
\eeq
we define the realization $|X_{*\centerdot}|$ of a bi-graded simplicial set $X_{*\centerdot}$ to be $|X_{*\centerdot}| = (|X_{ip}|) \in \Topstdot$.\\

We define the singular functor $S: \Topstdot \rarr \Sstdot$ entrywise. For $T_{*\cDot} \in \Topstdot$, $S(T_{*\cDot}) = (S(T_{ip}))$, where $S(T_{ip})$, $i,p=0,1$ is the simplicial set given by:
\beq
n \mapsto \Hom_{\Top}(|\Dn|, T_{ip}) \nonumber
\eeq
and $|\Dn|$ is the standard $n$-simplex:
\beq
|\Dn| = \{ (t_0, \cdots, t_n) \in \mathbb{R}^{n+1} \, | \, \sum_{i=0}^n t_i = 1 \, ; \, t_i \geq 0 \} \nonumber
\eeq
\begin{SLadj}
	For all $X_{*\cDot} \in \Sstdot$, $Y_{*\cDot} \in \Topstdot$, we have $\Hom_{\Topstdot}(|X_{*\cDot}|, Y_{*\cDot}) \cong \Hom_{\Sstdot}(X_{*\cDot}, SY_{*\cDot})$, that is $|\;| \dashv S$.
\end{SLadj}
\begin{proof}
It suffices to write:
\begin{align}
	\Hom_{\Topstdot}(|X_{*\cDot}|, Y_{*\cDot}) &= \oplus_{i,p = 0,1} \Hom_{\Top}(|X_{ip}|, Y_{ip}) \nonumber \\
	&\cong \oplus_{i,p = 0,1} \Hom_{\SetD}(X_{ip}, SY_{ip}) \nonumber \\
	& = \Hom_{\Sstdot}(X_{*\cDot}, SY_{*\cDot}) \nonumber
\end{align}
\end{proof}
We define a map $p_{*\cDot}: X_{*\cDot} \rarr Y_{*\cDot}$ of bi-graded simplicial sets to be a fibration if for $i,q=0,1$ the maps $p_{iq}: X_{iq} \rarr Y_{iq}$ are fibrations, i.e. if it is so entrywise. Thus a fibrant bi-graded simplicial set, or bi-graded Kan complex, is a bi-graded simplicial set $Y_{*\cDot}$ such that $Y_{ip} \rarr *$ is a fibration for $i,p=0,1$. In the same manner, a bi-graded continuous map $f_{*\cDot}: T_{*\cDot} \rarr U_{*\cDot}$ is said to be a bi-graded Serre fibration if it is so entry-wise. For $f_{*\cDot}, g_{*\cDot}: K_{*\cDot} \rarr X_{*\cDot}$ bi-graded simplicial maps, we say there is a bi-graded simplicial homotopy $f_{*\cDot} \xrightarrow{\simeq} g_{*\cDot}$ if we have a simplicial homotopy $f_{ip} \xrarr{\simeq} g_{ip}$ for $i,p=0,1$, which we recall means there is a commutative diagram:
\beq
\xymatrix{
K_{ip} \times \Delta^0 = K_{ip} \ar[d]_{1 \times d^1} \ar[dr]^{f_{ip}} \\
K_{ip} \times \Delta^1 \ar[r]^{h_{ip}} & X_{ip} \\
K_{ip} \times \Delta^0 \ar[u]^{1 \times d^0} \ar[ur]_{g_{ip}} } \nonumber
\eeq
in which case we say $h_{*\cDot} = (h_{ip})$ is a bi-graded simplicial homotopy $f_{*\cDot} \rarr g_{*\cDot}$. If $j_{*\cDot}: L_{*\cDot} \subset K_{*\cDot}$ denotes a bi-graded inclusion such that $f_{*\cDot}|_{L_{*\cDot}} = g_{*\cDot}|_{L_{*\cDot}}$, then we say we have a bi-graded simplicial homotopy $f_{*\cDot} \rarr g_{*\cDot}$ rel $L_{*\cDot}$ if we have a simplicial homotopy $f_{ip} \rarr g_{ip}$ rel $L_{ip}$ for $i,p=0,1$, that is such that the following diagrams commute:
\beq
\xymatrix{
K_{ip} \times \Delta^1 \ar[r]^-{h_{ip}} &X_{ip} \\
L_{ip} \times \Delta^1 \ar@{_{(}->}[u]^{j_{ip} \times 1} \ar[r]_-{pr_{L_{ip}}} &L_{ip} \ar[u]_{\alpha_{ip}}} \nonumber
\eeq
For $\Xstdot$ a fibrant bi-graded simplicial set, $v_{*\cDot} = (v_{ip}) \in (\Xstdot)_0$, we define $\pi_n(\Xstdot, v_{*\cDot}) = (\pi_n(X_{ip}, v_{ip}))$ where $\pi_n(X_{ip},v_{ip})$ is the set of homotopy classes of maps $\alpha_{ip}: \Dn \rarr X_{ip}$ rel $\partial \Dn$ for $i,p=0,1$, such that the following diagram is commutative:
\beq
\xymatrix{
\Dn \ar[r]^{\alpha_{ip}} &X_{ip} \\
\partial \Dn \ar@{^{(}->}[u] \ar[r] &\Delta^0 \ar[u]_{v_{ip}} } \nonumber
\eeq
and $\pi_0(\Xstdot) = (\pi_0(X_{ip}))$, $\pi_0(X_{ip})$ the set of path components of $X_{ip}$, $i,p=0,1$. Now a map $\fstdot: \Xstdot \rarr \Ystdot$ between fibrant bi-graded simplicial sets is said to be a weak equivalence if for all $\xstdot \in (\Xstdot)_0$, the induced map $(\fstdot)_*: \pi_k(\Xstdot, \xstdot) \rarr \pi_k(\Ystdot, \fstdot(\xstdot))$ is an isomorphism for all $k \geq 1$, that is $(f_{ip})_*: \pi_k(X_{ip}, x_{ip}) \xrightarrow{\cong} \pi_k(Y_{ip}, f_{ip}(x_{ip}))$ for $i,p=0,1$, and $(\fstdot)_*: \pi_0(\Xstdot) \rarr \pi_0(\Ystdot)$ is a bijection, that is $(f_{ip})_*: \pi_0(X_{ip}) \rarr \pi_0(Y_{ip})$ is a bijection for $i,p=0,1$, or in other terms a weak equivalence is so if it is a weak equivalence entry-wise.
\begin{fibr}
A map $\fstdot: \Xstdot \rarr \Ystdot$ between fibrant bi-graded simplicial sets is a trivial fibration if and only if $\fstdot$ has the right lifting property with respect to all maps $\partial \Delta^n \subset \Delta^n$ for $n \geq 0$, entry-wise.
\end{fibr}
\begin{proof}
The result holds entry-wise (\cite{GoJa}), hence is true for the corresponding bi-graded objects by definition.
\end{proof}
\begin{Eta}
Suppose $\Xstdot$ is a bi-graded Kan complex. Then the canonical map $\eta_{\Xstdot}: \Xstdot \rarr S|\Xstdot|$ is a weak equivalence.
\end{Eta}
\begin{proof}
The proof follows immediately from a similar proposition of \cite{GoJa} in the ungraded case because of the definitions of $S$ and $|\,\,|$, and the fact that $\eta_{\Xstdot}$ is a weak equivalence if it is so entry-wise, which is the case (\cite{GoJa}).
\end{proof}

Now if $\Xstdot$ is a bi-graded Kan complex, $\xstdot$ any vertex of $\Xstdot$, then by virtue of the above proposition we have:
\begin{align}
	\pi_n(\Xstdot, \xstdot) & \cong \pi_n(S|\Xstdot|, x_{* \cDot}) \nonumber \\
	&= \Hom(\Delta^2 \Dn/\pDn, S|\Xstdot| ) \nonumber \\
	&\cong \Hom( \Delta^2 |\Dn/\pDn|, |\Xstdot|) \nonumber \\
	&\cong \Hom(\Delta^2 S^n, |\Xstdot|) \nonumber \\
	&= \pi_n(|\Xstdot|, \xstdot) \,\, n \geq 1 \nonumber
\end{align}
so a map $\fstdot: \Xstdot \rarr \Ystdot$ of bi-graded Kan complexes is a weak equivalence if and only if the induced map $|\fstdot|: |\Xstdot| \rarr |\Ystdot|$ is a bi-graded topological weak equivalence, which leads us to defining, as in \cite{GoJa}, that a map $\fstdot: \Xstdot \rarr \Ystdot$ of bi-graded simplicial sets be a weak equivalence if the induced map $|\fstdot|: |\Xstdot| \rarr |\Ystdot|$ is a weak equivalence of bi-graded spaces. We define a cofibration of bi-graded simplicial sets to be an entry-wise inclusion.
\begin{SModStr}
$\Sstdot$ together with the classes of Kan fibrations, cofibrations and weak equivalences defined above is a model category.
\end{SModStr}
\begin{proof}
$\SetD$ is complete and cocomplete, hence so is $\Sstdot$. For the 2 out of 3 property, suppose we have $\fstdot = \gstdot \circ \hstdot$. If two of $\fstdot$, $\gstdot$ or $\hstdot$ is a weak equivalence, so are their entries, so writing $f_{ip} = g_{ip} \circ h_{ip}$ since the 2 out of 3 property holds at that level for $i,p=0,1$, then the third function would be an equivalence for $i,p=0,1$, hence so would be the resulting bi-graded function. For the retract property, say $\fstdot$ is a retract of $\gstdot$, and $\gstdot$ is a weak equivalence, fibration or cofibration. The retract breaks up into diagrams for $i,p=0,1$ as:
\beq
\xymatrix{
X \ar[d]^{f_{ip}} \ar[r] &Y \ar[d]^{g_{ip}} \ar[r] &X \ar[d]^{f_{ip}} \\
U \ar[r] &X \ar[r] &U} \nonumber
\eeq
$\gstdot$ having one of the three properties mentioned above, it is so entry-wise, at which level $f_{ip}$ being a retract of $g_{ip}$ the former map shares the same property, for $i,p=0,1$, hence $\fstdot$ shares that same property $\gstdot$ had. For the lifting property suppose we have:
\beq
\xymatrix{
U_{*\cDot} \ar[d]_{i_{*\cDot}} \ar[r] &\Xstdot \ar[d]^{p_{*\cDot}}\\
V_{*\cDot} \ar@{.>}[ur] \ar[r] &\Ystdot } \nonumber
\eeq
	with $i_{*\cDot}$ a cofibration, $p_{*\cDot}$ a fibration. We show there is a dotted arrow as shown making the diagram commutative if either $i_{* \cDot}$ or $p_{* \cDot}$ is a weak equivalence. That diagram breaks up into diagrams for $j,q=0,1$:
\beq
\xymatrix{
	U_{jq} \ar[d]_{i_{jq}} \ar[r] &X_{jq} \ar[d]^{p_{jq}}\\
	V_{jq} \ar@{.>}[ur] \ar[r] &Y_{jq} } \nonumber
\eeq
	Suppose for argument's sake, $p_{*\cDot}$ is a weak equivalence. Then it is so entry-wise, so in the diagrams above for $j,q=0,1$ there is a dotted arrow, hence the original bi-graded square has a diagonal arrow as claimed. Finally for the functorial factorization property, let $\fstdot: \Xstdot \rarr \Ystdot$ be a map in $\Sstdot$. We will show half of the property, namely that $\fstdot$ can factor as $p_{*\cDot} \circ i_{*\cDot}$ with $p_{*\cDot}$ a trivial fibration, $i_{*\cDot}$ a cofibration, the other factorization as a trivial cofibration followed by a fibration being proved in like manner. $\fstdot = (f_{ip})$, and entry-wise the factorization property holds, so we can write: $f_{jq} = p_{jq} \circ i_{jq}$, with $p_{jq}$ a trivial fibration for $j,q=0,1$, $i_{jq}$ a cofibration for $j,q=0,1$. Writing $i = (i_{jq})$ and $p = (p_{jq})$ we have $\fstdot = p_{*\cDot} \circ i_{*\cDot}$ with $p_{*\cDot}$ a trivial fibration, and $i_{*\cDot}$ a cofibration. This completes the proof.
\end{proof}

\subsubsection{Model category structure on $\sksMod_*$}
We adapt Thm II.4.1 of \cite{GoJa} to our setting. Let $\cCstdot$ be a $\mathbb{Z}_2$-bi-graded category, $s\cCstdot$ the category of simplicial objects in $\cCstdot$, $s\cCstdot= \cCstdot^{\Dop}$. We assume there is a functor $\Gstdot: s\cCstdot \rarr \Sstdot$ with a left adjoint $F_{*\cDot}: \Sstdot \rarr s\cCstdot$. We define a morphism $\fstdot: M_{*\cDot} \rarr N_{*\cDot}$ in $\sksMod_*$ to be a weak equivalence (resp. a fibration) if $\Gstdot\fstdot$ is a weak equivalence (resp. a fibration) of bi-graded simplicial sets, and $\fstdot$ is a cofibration if it has the left lifting property with respect to all trivial fibrations in $s\cCstdot$. Entry-wise, that gives us $G_{ip}: s\cC \rarr \SetD$ with a left adjoint $F_{ip}: \SetD \rarr s\cC$. Given the model category structure we put on $\Sstdot$, fibrations, cofibrations and weak equivalences in $\scCstdot$ are defined entry-wise in $s\cC$. We will apply this formalism to the case $\cC = \kMod$, so that $s\cC = \skMod$ and $s\cCstdot = \sksMod_*$.\\

Note also that we have a natural map:
\beq
\colim_I \Gstdot(X_{\alpha}) \rarr \Gstdot( \colim_I X_{\alpha}) \nonumber
\eeq
which decomposes into entry maps:
\beq
\colim_I G_{ip}(X_{\alpha, ip}) \rarr G_{ip}( \colim_I X_{\alpha, ip}) \nonumber
\eeq
for $i,p=0,1$.

\newpage

\begin{sCdMod}
Suppose $\cC$ is bicomplete and $\Gstdot: \scCstdot \rarr \Sstdot$ commutes with filtered colimits. Then with the classes of fibrations, cofibrations and weak equivalences defined above, along with the assumption that a cofibration with the left lifting property with respect to all fibrations be a weak equivalence, $s\cCstdot$ is a model category.
\end{sCdMod}
\begin{proof}
Since $\Gstdot$ commutes with filtered colimits, it does so entry-wise, so the hypotheses of Thm II.4.1 of \cite{GoJa} are met, hence $s\cC$ is a model category, something we will use to prove $\scCstdot$ itself is a model category. $\scCstdot$ is clearly bicomplete. For the 2 out of 3 property, suppose we have a factorization in $\scCstdot$:
\beq
\xymatrix{
\Xstdot \ar[dr]_{\fstdot} \ar[rr]^{\gstdot} &&\Ystdot \\
	&Z_{*\cDot}  \ar[ur]_{\hstdot}} \nonumber
\eeq
This breaks up into diagrams for $i,p=0,1$:
\beq
\xymatrix{
X_{ip} \ar[rd]_{f_{ip}} \ar[rr]^{g_{ip}} &&Y_{ip} \\
&Z_{ip} \ar[ur]_{h_{ip}}  } \nonumber
\eeq
	in $s\cC$, where the 2 out of 3 property holds since it is a model category. Suppose for illustrative purposes $\gstdot$ and $\fstdot$ are weak equivalences, then $g_{ip}$ and $f_{ip}$ are so for $i,p=0,1$, hence $h_{ip}$ is a weak equivalence for $i,p=0,1$, thus so is $\hstdot = (h_{ip})$. For the retract property, $\fstdot$ a retract of $\gstdot$ is a property defined by a diagram in $\scCstdot$, which breaks up into retract diagrams in $s\cC$, where the retract property holds, so starting from $\gstdot$ with a property $P$, it being a weak equivalence, a fibration, or a cofibration, it is so entry-wise, so by the retract property in $s\cC$ $f_{ip}$ shares the same property for $i,p=0,1$, and those entries recombine into $\fstdot$ with that same property $P$ $\gstdot$ had. For the lifting property and the functorial factorization property, the argument is the same, we work entry-wise and use the fact that $s\cC$ is a model category from \cite{GoJa}.
\end{proof}
\begin{sksModModcat}
With the above notions of fibrations, cofibrations and weak equivalences, $\sksMod_*$ is a model category.
\end{sksModModcat}
\begin{proof}
	Let $\cC = \kMod$, so that $s\cC = \skMod$ and $\scCstdot = \sksMod_*$. We know $\cC$ is bicomplete. Start from the forgetful functor $G: \kMod \rarr \Set$, with a left adjoint $F$, which we both prolong to the simplicial case to get maps which we will again denote by $G$ and $F$: $G: \skMod \rarr \SetD$ and its left adjoint $F$, by defining $G(X)_n = G(X_n)$. $G$ preserving filtered colimits, so will its bi-graded generalization $\Gstdot: \sksMod_* \rarr \Sstdot$, which has a bi-graded generalization of $F$ for left adjoint, denoted $F_{*\cDot}$. At this point we just use the previous theorem.
\end{proof}

\subsubsection{bi-graded simplicial categories}
About notations, if $\cC$ is a bi-graded category, write $\underline{\cC}$ for its ungraded counterpart. For instance if $\cC = \sksMod_*$, then $\underline{\cC} = \skMod$, which generically refers to either $\skMod_0$ or $\skMod_1$. Also, for ease of reading, we will just write $M$ for $M_{*\cDot} \in \ksMod_*$. We also define the enhanced $\Hom$ set $\Hom^+$ as follows: $(\Hom^+(M,N))_0 = \Hom(M,N)$, while $(\Hom^+(M,N))_1$ consist of the set of parity reversing morphisms from $M$ to $N$.
\begin{grscat}
A bi-graded category $\cC$ is a bi-graded simplicial category, following the ungraded definition in \cite{GoJa}, if there is a mapping space functor:
\beq
\uHomC(-,-):\cC^{\op} \times \cC \rarr \Sstdot \nonumber
\eeq
such that $\forall M,N \in \Ob(\cC)$:
\begin{enumerate}
	\item $\uHomC(M,N)_{0\times 0} = \Hom^+_{\cC}(M,N)$
  	\item $\uHomC(M,-): \cC \rarr \Sstdot$ has a left-adjoint $M \boxtimes - :\Sstdot \rarr \cC$. Adjointness means:
\beq
		\Hom_{\cC}(M \boxtimes K , N) \cong \HomSstdot(K, \uHomC(M,N))  \nonumber
\eeq
with associativity $M \boxtimes (K \times L) \cong (M \boxtimes K) \boxtimes L$.
  	\item $- \boxtimes K: \cC \rarr \cC$ has a right adjoint $\exp_{-}(K) = \uhomC(K,-): \cC \rarr \cC$ i.e.
\beq
		\Hom_{\cC}(M \boxtimes K, N) \cong \Hom_{\cC}(M, N^K)  \nonumber
\eeq
\end{enumerate}
\end{grscat}

\begin{sksModgrscat}
$s\ucC = \skMod$ being a simplicial category (\cite{TV4}), it follows that $s\cC = \sksMod_*$ becomes a bi-graded simplicial category with:
	\begin{align}
		\uHomsC(M,N)_{n \times n} &= \HomsC( M \boxtimes \Delta^2 \Delta^n, N) \nonumber \\
		&= (\HomsuC(M_{ip}, N_{ip})_n) \nonumber
	\end{align}
\end{sksModgrscat}
\begin{proof}
Here, as in \cite{TV}:
\beq
- \otimes -: \skMod \times \SetD \rarr \skMod \nonumber
\eeq
is defined by:
\beq
(M \otimes K)_n = \coprod_{k \in K_n} M_n \nonumber
\eeq
For $\psi: [m] \rarr [n]$ in $\Delta$, we have an induced map $\psi^*: (M \otimes K)_m \rarr (M \otimes K)_n $ that comes from:
\beq
\coprod_{k \in K_m} M_m \rarr \coprod_{k \in K_m} M_n \rarr \coprod_{k \in K_n} M_n \nonumber
\eeq
We determine $\HomsC(M \boxtimes K, N)$. In a first time, $\skMod$ being a simplicial category, if $M,N \in \skMod$, $K \in \SetD$, we have $\Hom(M \oT K, N) \cong \Hom(K, \uHom_{\skMod}(M,N))$. If now $M,N \in \sksMod$, $K \in \Sd = (\SetD)_0 \oP (\SetD)_1$:
\begin{align}
	\Hom(M  \oT K, N) &= \Hom((M \oT K)_0 \oP (M \oT K)_1, N_0 \oP N_1) \nonumber \\
		&=\Hom((M_0 \oT  K_0)  \oP (M_1 \oT K_1),N_0) \nonumber \\
		&\qquad \oP \Hom ((M_0 \oT K_1) \oP (M_1 \oT K_0), N_1) \nonumber \\
		&=\Hom(M_0 \oT K_0,N_0) \oP \Hom( M_1 \oT K_1,N_0) \nonumber \\
		&\qquad \oP \Hom(M_0 \oT K_1,N_1) \oP \Hom(M_1 \oT K_0,N_1) \nonumber \\
		&\cong \Hom_{\SetD}(K_0, \uHom_{\skMod}(M_0,N_0)) \oP \Hom(K_1,\uHom(M_1,N_0)) \nonumber \\
		&\qquad \oP \Hom (K_1,\uHom(M_0,N_1)) \oP \Hom(K_0, \uHom(M_1,N_1)) \nonumber \\
		&=\Hom(K_0, \uHom(M_0,N_0) \oP \uHom(M_1,N_1)) \nonumber \\
		&\qquad \oP \Hom(K_1, \uHom(M_1,N_0) \oP \uHom(M_0,N_1)) \nonumber \\
		&= \Hom_{S_{\cDot}}(K, \uHom_{\sksMod}(M,N)) \nonumber
\end{align}
if we define:
	\begin{align}
		(\uHom_{\sksMod}(M,N))_0 &= \uHom(M_0,N_0) \oP \uHom(M_1,N_1) \nonumber \\
		(\uHom_{\sksMod}(M,N))_1 &= \uHom(M_0,N_1) \oP \uHom(M_1,N_0) \nonumber
	\end{align}
Now for $M,N \in \sksMod_* = s\cC$, $K \in \Sstdot$, we have:
\begin{align}
	\HomsC(M \bT K, N) &= \oP_i \Hom_{\sksMod}(M_i \oT K_i, N_i) \nonumber \\
	&\cong \oP_i \Hom_{\Sd}(K_i, \uHom_{\sksMod}(M_i,N_i)) \nonumber \\
	&= \Hom_{\Sstdot}(K, \uHom_{\sksMod_*}(M,N)) \label{MbTKNHom}
\end{align}
if we define:
\beq
	\uHomsC(M,N) = \oP_i \uHom_{\sksMod}(M_i,N_i) \nonumber
\eeq
Thus \eqref{MbTKNHom} shows $M \bT - \dashv \uHom_{s\cC}(M,-)$. We also have:
\begin{align}
	\uHomsC(M,N)_{ 0 \times 0} & = \Hom_{\Sstdot}(\Delta^2 \Delta^0,\uHomsC(M,N)) \nonumber \\
	&=\oP_i \Hom_{\Sd}(\Delta \Delta^0, \uHom_{\sksMod}(M_i,N_i)) \nonumber \\
	&=\oP_i \{\Hom_{\SetD}(\Delta^0, \oP_p \uHom_{\skMod}(M_{ip},N_{ip})) \nonumber \\
	&\qquad \oP \Hom(\Delta^0, \oP_{p \neq q} \uHom_{\skMod}(M_{ip},N_{iq}))\} \nonumber \\
	&=\oP_i \{ \oP_p \Hom_{\SetD}(\Delta^0,\uHom_{\skMod}(M_{ip},N_{ip})) \nonumber \\
	& \qquad\oP \oP_{p \neq q} \Hom(\Delta^0, \uHom_{\skMod}(M_{ip}, N_{iq}))\} \nonumber \\
	&= \oP_i \{ \oP_p( \uHom_{\skMod}(M_{ip}, N_{ip}))_0 \oP \oP_{p \neq q} (\uHom(M_{ip}, N_{iq}))_0 \} \nonumber \\
	&= \oP_i \{ \oP_p \HomsuC(M_{ip},N_{ip}) \oP \oP_{p \neq q} \HomsuC(M_{ip}, N_{iq}) \} \nonumber \\
	&=\oP_i \{ \Hom^+_{\sksMod}(M_i,N_i) \} \nonumber \\
	&= \HomsC^+(M,N) \nonumber 
\end{align}
From \eqref{MbTKNHom} again, we have:
\begin{align}
	\uHomsC(M,N)_{ n \times n} &= \Hom_{\Sstdot}(\Delta^2 \Dn, \uHom_{\sksMod_*}(M,N)) \nonumber \\ 
&\cong \HomsC(M \boxtimes \Delta^2 \Delta_n, N) \nonumber
\end{align}
as claimed. For the associativity, we use the associativity of $\otimes$ in the ungraded case:
\begin{align}
	(M \boxtimes K) \boxtimes L &= \oP_i (M_i \oT K_i) \oT L_i \nonumber \\
	&=\oP_i \oP_{p,q,r} (M_{ip} \oT K_{iq}) \oT L_{ir} \nonumber \\
	&\cong \oP_i \oP_{p,q,r} M_{ip} \oT (K_{iq} \times L_{ir}) \nonumber \\
	&=\oP_i(\oP_p M_{ip}) \times (\oP_{q,r} K_{iq} \times L_{ir}) \nonumber \\
	&=\oP_i M_i \oT (K_i \times L_i) \nonumber \\
	&=\oP_i M_i \oT (K \times L)_i \nonumber \\
	&= M \bT (K \times L) \nonumber
\end{align}
For the exponent map:
\begin{align}
	\HomsC(M \boxtimes K, N) &= \oP_i \Hom_{\sksMod}(M_i \oT K_i, N_i) \nonumber \\
	&=\oP_i \{ \Hom_{\skMod}(\oP_p M_{ip} \oT K_{ip}, N_{i0}) \nonumber \\
	&\qquad \oP \Hom(\oP_{p \neq q} M_{ip} \oT K_{iq},N_{i1}) \} \nonumber \\
	&= \oP_i \{ \oP_p \HomsuC(M_{ip} \oT K_{ip},N_{i0}) \nonumber \\
	&\qquad \oP \oP_{p \neq q} \HomsuC(M_{ip} \oT K_{iq},N_{i1}) \} \nonumber \\
	&\cong \oP_i \{ \oP_p \HomsuC(M_{ip},N_{i0}^{K_{ip}}) \oP \oP_{p \neq q} \HomsuC(M_{ip}, N_{i1}^{K_{iq}}) \} \nonumber \\
	&= \oP_i \{ \HomsuC(M_{i0}, N_{i0}^{K_{i0}}) \oP \HomsuC(M_{i1}, N_{i0}^{K_{i1}}) \nonumber \\
	&\qquad \oP \HomsuC(M_{i0}, N_{i1}^{K_{i1}}) \oP \HomsuC(M_{i1}, N_{i1}^{K_{i0}}) \} \nonumber \\
	&= \oP_i \{ \HomsuC(M_{i0}, (N_i^{K_i})_0) \oP \HomsuC(M_{i1}, (N_i^{K_i})_1) \} \nonumber 
\end{align}
if we define:
\begin{align}
	(N_i^{K_i})_0 &= N_{i0}^{K_{i0}} \oP N_{i1}^{K_{i1}} \nonumber \\
	(N_i^{K_i})_1 &= N_{i0}^{K_{i1}} \oP N_{i1}^{K_{i0}} \nonumber 
\end{align}

From there
\begin{align}
	\HomsC(M \bT K, N) &= \oP_i \Hom_{\sksMod}(M_i, N_i^{K_i}) \nonumber \\
	&= \HomsC(M,N^K) \nonumber
\end{align}
if we define $N^K = \oP_i N_i^{K_i}$. This shows $- \bT K \dashv \exp_{-}(K)$. This completes the proof.
\end{proof}

\subsubsection{Bi-graded simplicial model categories}
We axiomatize the definition of bi-graded simplicial model category as done in \cite{GoJa} in the ungraded case. We first need to define pullbacks in $\Sstdot$. It being a bi-graded category, pullbacks are defined entrywise. In what follows $\cC = \sksMod_*$. We also use the abbreviation $\uHomC(X,Y)_{ip} = XY_{ip}$ for $i,p=0,1$. Consider the cartesian square:
\beq
\xymatrix{
	\Gamma = \uHomC(A,X) \times_{\uHomC(A,Y)} \uHomC(B,Y) \ar[d] \ar[r] &\uHomC(B,Y) = (BY_{ip}) \ar[d]\\
 \uHomC(A,X)=(AX_{ip}) \ar[r] & \uHomC(A,Y)=(AY_{ip}) 
 } \nonumber
\eeq
it breaks up into individual cartesian squares:
\beq
\xymatrix{
	\Gamma_{ip} \ar[d] \ar[r] &BY_{ip} \ar[d] \\
AX_{ip} \ar[r] &AY_{ip} } \nonumber
\eeq
for $i,p=0,1$, giving $\Gamma = \oP_{i,p} AX_{ip} \times_{AY_{ip}} BY_{ip} = ( AX_{ip} \times_{AY_{ip}} BY_{ip})$. 

\begin{grSimpModCat}
Let $\cC$ be a bi-graded model category and a bi-graded simplicial category. Suppose $j:A \rarr B$ is a cofibration, $q:X \rarr Y$ a fibration. Then:
\beq
\uHomC(B,X) \xrightarrow{(j^*,q_*)} \uHomC(A,X)\times_{\uHomC(A,Y)} \uHomC(B,Y) \nonumber
\eeq
is a fibration in $\Sstdot$, which is trivial if either of $j$ or $q$ is.
\end{grSimpModCat}
\begin{grsModCat}
A category satisfying the axiom grsModCat above will be called a bi-graded simplicial model category
\end{grsModCat}
This definition follows exactly the definition of such categories in the ungraded case as laid out in \cite{GoJa}. We prove a preliminary result found in the same reference, that will be instrumental in proving that $\sksMod_*$ is a bi-graded simplicial model category.
\begin{PropgrsModCat}
Let $\cC$ be a bi-graded model category and a bi-graded simplicial category, $i:K \rarr L$ a cofibration in $\Sstdot$, $q:X \rarr Y$ a fibration in $\cC$. Then the grsModCat axiom is equivalent to:
\beq
\uhomC(L,X) \rarr \uhomC(K,X) \times_{\uhomC(K,Y)}\uhomC(L,Y) \nonumber
\eeq
being a fibration, trivial if either of $i$ or $q$ is. Here we have denoted $N^K = \uhomC(K,N)$ for ease of reading.
\begin{proof}
It suffices to work entry-wise. We have:
\beq
	\uHomC(B,X)_{ip} \rarr \uHomC(A,X)_{ip} \times_{\uHomC(A,Y)_{ip}} \uHomC(B,Y)_{ip} \nonumber\\
\eeq
a fibration. For $p = 0$, we have a fibration:
\beq
\xymatrix{
		\uHom(B_{i0},X_{i0}) \oP \uHom(B_{i1},X_{i1}) \ar[d] \\
		\oP_p \uHom(A_{ip},X_{ip}) \times_{\oP_p \uHom(A_{ip},Y_{ip})} \oP_p \uHom(B_{ip},Y_{ip}) \ar@{=}[d] \\
		\oP_p \uHom(A_{ip},X_{ip}) \times_{\uHom(A_{ip},Y_{ip})} \uHom(B_{ip},Y_{ip}) 
		} \nonumber
\eeq
which decomposes into:
\beq
	\uHom(B_{ip},X_{ip}) \xrarr{(j_{ip}^*,(q_{ip})_*)} \uHom(A_{ip},X_{ip}) \times_{\uHom(A_{ip},Y_{ip})} \uHom(B_{ip},Y_{ip}) \nonumber
\eeq
	for $p = 0,1$, fibration in $\SetD$, trivial if either of $j$ or $q$ is, in particular if $j_{ip}$ or $q_{ip}$ is. Using the ungraded counterpart of the proposition from \cite{GoJa}, this is equivalent to:
	\beq
	\uhomuC(L_{ip},X_{ip}) \rarr \uhomuC(K_{ip}, X_{ip}) \times_{\uhomuC(K_{ip},Y_{ip})} \uhomuC(L_{ip},Y_{ip}) \label{SUSYhomGoJa1}
	\eeq
being a fibration, trivial if either $i_{ip}: K_{ip} \rarr L_{ip}$ or $q_{ip}: X_{ip} \rarr Y_{ip}$ is, for $p = 0,1$. \\

For $p = 1$, we have:
\beq
\xymatrix{
		\uHom(B_{i0},X_{i1}) \oP \uHom(B_{i1},X_{i0}) \ar[d] \\
		\oP_{p \neq q} \uHom(A_{ip},X_{iq}) \times_{ \oP_{p \neq q} \uHom(A_{ip}, Y_{iq})} \uHom(B_{ip},Y_{iq}) \ar[d] \\
	\oP_{p \neq q} \uHom(A_{ip},X_{iq}) \times_{\uHom(A_{ip},Y_{iq})} \uHom(B_{ip},Y_{iq}) 
		}\nonumber
\eeq
which decomposes as:
\beq
	\uHom(B_{i0},X_{i1}) \xrarr{(j_{i0}^*,(q_{i1})_*)} \uHom(A_{i0},X_{i1}) \times_{\uHom(A_{i0},Y_{i1})} \uHom(B_{i0},Y_{i1}) \nonumber
\eeq
and:
\beq
	\uHom(B_{i1},X_{i0}) \xrarr{(j_{i1}^*,(q_{i0})_*)} \uHom(A_{i1},X_{i0}) \times_{\uHom(A_{i1},Y_{i0})} \uHom(B_{i1},Y_{i0}) \nonumber
\eeq
	both fibrations, and trivial if $j_{ik}$, or $q_{il}$ trivial with $k \neq l$. Again, by the ungraded counterpart of this result from \cite{GoJa}, this is equivalent to:
	\beq
	\uhomuC(L_{ik},X_{il}) \rarr \uhomuC(K_{ik},X_{il}) \times_{\uhomuC(K_{ik},Y_{il})} \uhomuC(L_{ik},Y_{il}) \label{SUSYhomGoJa2}
	\eeq
	fibration, trivial if $i_{ik}: K_{ik} \rarr L_{ik}$ or $q_{il}: X_{il} \rarr Y_{il}$ is, for $k \neq l$. Recombining \eqref{SUSYhomGoJa1} and \eqref{SUSYhomGoJa2} yields:
	\beq
	\uhomC(L,X) \rarr \uhomC(K,X) \times_{\uhomC(K,Y)} \uhomC(L,Y) \nonumber
	\eeq
	fibration, trivial if $i:K \rarr L$ or $q:X \rarr Y$ is.
\end{proof}
\end{PropgrsModCat}
\begin{sksModgrsModCat}
$\sksMod_*$ is a bi-graded simplicial model category.
\end{sksModgrsModCat}
\begin{proof}
We use the functor $G: \cC = \sksMod_* \rarr \Sstdot$ used to put a model structure on $\sksMod_*$. We wish to show $\sksMod_*$ satisfies the grsModCat axiom, which we just showed is equivalent to saying in particular the map
\beq
	\uhomC(L,X) \rarr \uhomC(K,X) \times_{\uhomC(K,Y)} \uhomC(L,Y) \nonumber
\eeq
	is a fibration for $i:K \rarr L$ a cofibration in $\Sstdot$, $q:X \rarr Y$ a fibration in $\cC$ (trivial if either of $i$ or $q$ is), which means $G(X^L) \rarr G(X^K \times_{Y^K} Y^L)$ is a fibration in $\Sstdot$ (trivial if either of $i$ or $q$ is). $G$ being a right adjoint it commutes with finite limits so this is equivalent to showing that:
\beq
G(X^L) \rarr G(X^K) \times_{G(Y^K)} G(Y^L) \nonumber
\eeq
is a fibration, trivial if either of $i$ or $q$ is. Now $F \dashv G: \Sstdot \rarr \cC$ satisfies, for $L,K \in \Sstdot$:
\begin{align}
	F(L \boxtimes K) &= F(\oP_i L_i \oT K_i)\nonumber \\
	&= \oP_i F_i(L_i \oT K_i)\nonumber \\
	&=\oP_i F_i(\oP_{p,q} L_{ip} \oT K_{iq}) \nonumber \\
	&=\oP_i [F_{i0}((L_{i0} \oT K_{i0}) \oP (L_{i1} \oT K_{i1})) \nonumber \\
	&\qquad \oP F_{i1}((L_{i0} \oT K_{i1}) \oP (L_{i1} \oT K_{i0}))] \nonumber \\
	&= \oP_i \big( F_{i0}(L_{i0} \oT K_{i0}) \oP  F_{i0}(L_{i1} \oT K_{i1}) \nonumber \\
	&\qquad \oP  F_{i1}(L_{i0} \oT K_{i1}) \oP F_{i1}(L_{i1} \oT K_{i0}) \big) \nonumber \\
	&\cong \oP_i \big( (F_{i0}L_{i0} \oT K_{i0})  \oP (F_{i1}L_{i1} \oT K_{i1}) \nonumber \\
	&\qquad \oP  (F_{i0}L_{i0} \oT K_{i1}) \oP (F_{i1}L_{i1} \oT K_{i0}) \big) \nonumber \\
	&=\oP_i \big( [(F_iL_i)_0 \oT K_{i0}] \oP   [(F_iL_i)_1 \oT K_{i1}] \nonumber \\
	&\qquad \oP [(F_iL_i)_0 \oT K_{i1}] \oP [(F_iL_i)_1 \oT K_{i0}]  \big) \nonumber \\
	&=\oP_i \oP_{p,q} (FL)_{ip} \oT K_{iq} \nonumber \\
	&=FL \bT K
\end{align}
where we have used the fact that $F_{ip} \dashv G_{ip}$, $i,p=0,1$, so it commutes with colimits as a left adjoint. Note that we used (for $p = q + r$):
\beq
	F_{ip}(L_{iq} \oT K_{ir}) \cong F_{iq}L_{iq} \oT K_{ir} \nonumber
\eeq
throughout. We need the following lemma:
\begin{Ghom}
If for any $K,L \in \Sstdot$ there is a natural isomorphism $F(L \boxtimes K) \cong F(L) \boxtimes K$, then for any $X \in \cC = \sksMod_*$:
\beq
G\uhomC(K,X) \cong \uhomSstdot(K, G(X)) \nonumber
\eeq
\end{Ghom}
\begin{proof}
It suffices to write:
\begin{align}
	\HomSstdot(L, &G\uhomC(K,X)) = \oP_{i,p} \HomS(L_{ip}, G_{ip} \uhomuC(K,X)_{ip}) \nonumber \\
	&=\oP_i [\HomS(L_{i0},G_{i0}(X_{i0}^{K_{i0}} \oP X_{i1}^{K_{i1}})) \nonumber \\
	&\qquad \oP \HomS(L_{i1},G_{i1}(X_{i0}^{K_{i1}} \oP X_{i1}^{K_{i0}}))] \nonumber \\
	&=\oP_i [\HomS(L_{i0},G_{i0}X_{i0}^{K_{i0}}) \oP \HomS(L_{i0},G_{i0}X_{i1}^{K_{i1}}) \nonumber \\
	&\qquad \oP \HomS(L_{i1},G_{i1}X_{i0}^{K_{i1}}) \oP \HomS(L_{i1},G_{i1}X_{i1}^{K_{i0}})] \nonumber \\
	&\cong \oP_i [\HomuC(F_{i0}L_{i0},X_{i0}^{K_{i0}}) \oP \HomuC(F_{i0}L_{i0},X_{i1}^{K_{i1}}) \nonumber \\
	&\qquad \oP \HomuC(F_{i1}L_{i1},X_{i0}^{K_{i1}}) \oP \HomuC(F_{i1}L_{i1},X_{i1}^{K_{i0}})] \nonumber \\
	&\cong \oP_i [\HomuC(F_{i0}L_{i0} \oT K_{i0},X_{i0}) \oP \HomuC(F_{i0}L_{i0} \oT K_{i1},X_{i1}) \nonumber \\
	&\qquad \oP \HomuC(F_{i1}L_{i1} \oT K_{i1},X_{i0}) \oP \HomuC(F_{i1}L_{i1} \oT K_{i0},X_{i1})] \nonumber \\
	&=\oP_i [\HomuC((FL \bT K)_{i0},X_{i0}) \oP \HomuC((FL \bT K)_{i1}, X_{i1})] \nonumber \\
	&=\oP_i \Hom_{\sksMod}((FL \bT K)_i, X_i) \nonumber \\
	&=\Hom_{\cC}(FL \bT K,X) \nonumber \\
	&\cong \Hom_{\cC}(F(L \bT K),X) \nonumber \\
	&\cong \Hom_{\cC}(L \bT K, GX) \nonumber \\
	&= \oP_i [\HomuC(L_{i0} \oT K_{i0},(GX)_{i0}) \oP \HomuC(L_{i1} \oT K_{i1},(GX)_{i0}) \nonumber \\
	&\qquad \oP \HomuC(L_{i0} \oT K_{i1},(GX)_{i1}) \oP \HomuC(L_{i1} \oT K_{i0},(GX)_{i1})] \nonumber \\
	&\cong \oP_i [\HomS(L_{i0},\uhom_{\SetD}(K_{i0},(GX)_{i0})) \nonumber \\
	&\qquad \oP \HomS(L_{i1},\uhom_{\SetD}(K_{i1},(GX)_{i0})) \nonumber \\
	&\qquad \oP \HomS(L_{i0},\uhom_{\SetD}(K_{i1},(GX)_{i1})) \nonumber \\
	&\qquad \oP \HomS(L_{i1},\uhom_{\SetD}(K_{i0},(GX)_{i1}))] \nonumber  \\ 
	&=\oP_i \HomSd(L_i, \uhomSstdot(K,GX)_i) \nonumber \\
	&=\HomSstdot(L, \uhomSstdot(K, GX))\nonumber
\end{align}
with $S = \SetD$.
\end{proof}

\newpage

From there,
\beq
G(X^L) \rarr G(X^K) \times_{G(Y^K)} G(Y^L) \nonumber
\eeq
that is:
\beq
G \uhomC(L,X) \rarr G\uhomC(K,X) \times_{G \uhomC(K,Y)} G \uhomC(L,Y) \nonumber
\eeq
is equivalent to:
\beq
\uhomSstdot(L, GX) \rarr \uhomSstdot(K, GX) \times_{\uhomSstdot(K,GY)} \uhomSstdot(L,GY) \nonumber
\eeq
Now $q:X \rarr Y$ fibration in $\cC$ means $G(q): GX \rarr GY$ fibration in $\Sstdot$, and $\Sstdot$ being a bi-graded simplicial model category it satisfies the grsModCat axiom, so the above map is a fibration, trivial if either of $q$ or $i$ is, that is $\sksMod_*$ is a bi-graded simplicial model category.
\end{proof}

\subsubsection{Internal hom}
From \cite{CCF} and \cite{V} we know $\ksMod$ has an internal hom $\ucHom _{\ksMod}$. This comes from the fact that the tensor product $\oT_k$ on $\ksMod$ involves terms of mixed parity. In those references the internal hom is defined as follows:
\beq
\ucHom_{\ksMod}(M,N)_0 = \Hom_{\ksMod}(M,N) = \Hom_k(M,N)\nonumber
\eeq
while $\ucHom_{\ksMod}(M,N)_1$ is the set of morphisms $\phi: M \rarr N$ that reverse parity, i.e. those morphisms $M_0 \rarr N_1$ and $M_1 \rarr N_0$. We can see this is indeed the correct definition by starting from the formal definition for the internal hom:
\beq
\Hom_k(M \otimes_k N , P) \cong \Hom_k(M, \ucHom_{\ksMod}(N,P)) \nonumber
\eeq
Expanding the left hand side in full:
\begin{align}
\Hom(M_0\oT N_0,P_0) &\oP \Hom(M_1 \oT N_1,P_0) \nonumber \\
	&\oP \Hom(M_0 \oT N_1,P_1) \oP \Hom( M_1 \oT N_0,P_1) \nonumber \\
	\cong \Hom(M_0,& \Hom_k(N_0,P_0)) \oP \Hom(M_1,\Hom_k(N_1,P_0)) \nonumber \\
	&\oP \Hom(M_0,\Hom_k(N_1,P_1)) \oP \Hom(M_1,\Hom_k(N_0,P_1)) \nonumber
\end{align}
where we used the fact that $\kMod$ is a closed monoidal category in the sense of \cite{Ho}, with internal hom $\Hom_k$. It follows as claimed above that:
\begin{align}
	\ucHom(N,P)_0 &= \Hom(N_0,P_0) \oP \Hom(N_1,P_1) \nonumber \\
	\ucHom(N,P)_1 &= \Hom(N_1,P_0) \oP \Hom(N_0,P_1) \nonumber
\end{align}
Note that this shows what we called the enhanced hom $\Hom^+$ earlier is actually the internal hom: $\Hom^+ = \ucHom$. Henceforward we will use the notation $\ucHom$.\\

For the tensor product in $\ksMod_*$:
\begin{align}
	\Hom(M \bT N, P) &= \oPi \Hom(M_i \oT N_i, P_i) \nonumber \\
	& \cong  \oPi \Hom(M_i, \ucHom_{\ksMod}(N_i,P_i))\nonumber \\
	& = \Hom(M, \ucHom_{\ksMod_*}(N,P))\nonumber
\end{align}
if $\ucHom_{\sksMod_*}(N,P) = \oPi \ucHom_{\sksMod}(N_i,P_i)$. We prolong this to the simplicial case levelwise in the simplicial index:
\beq
\xymatrix{
\Hom_{\ksMod_*}(M_n \bT N_n, P_n) \ar@{~>}[d] \ar[r]^-{\cong} & \Hom(M_n, \ucHom_{\ksMod_*}(N_n,P_n)) \ar@{~>}[d] \nonumber \\
\Hom_{\sksMod_*}(M \bT N, P) \ar@{.>}[r]_-{\cong} & \Hom(M, \ucHom_{\sksMod_*}(N,P))
}
\eeq
At this point $\cC=\sksMod_*$ the category of simplicial graded $k$-supermodules is endowed with the level-wise tensor product $\bT$ for which we have an internal hom $\ucHom_{\sksMod_*}$, and is endowed with the "usual" model structure whereby weak equivalences and fibrations are defined on the underlying bi-graded simplicial sets. We have also proved $\cC$ is a bi-graded simplicial model category with:
\beq
\uHomC(M,N)_{n\times n} \cong \Hom_{\cC}(M \boxtimes \Delta^2 \Delta^n, N) \nonumber
\eeq
Let:
\beq
\Comm(\cC) = \sksAlg_* \nonumber
\eeq
where again commutativity is defined on superalgebras by $ab = (-1)^{|a||b|}ba$ on homogeneous elements.\\

Now for $A \in \sksAlg_*$, denote by $\AsMod_*$ the category of objects of $\sksMod_*$ that are $A$-modules, with $(\AsMod_*)_n = A_n \text{-sMod}$. A morphism $f$ of $A$-supermodules is a simplicial morphism of supermodules $f:M \rarr N$ such that $f(am) = af(m)$ for $a \in A$, $m \in M$, $M \in \AsMod_*$. To be more specific $f$ has components $f_n: M_n \rarr N_n$ in $\Hom_{A_n}(M_n, N_n)$ and:
\begin{align}
	f(am)&= \oP_n f_n(a_n m_n) \nonumber \\
	&=\oP_n a_n f_n(m_n) \nonumber \\
	&=af(m) \nonumber
\end{align}

There is a tensor product $M \tbTA N$ defined level-wise:
\beq
(M \tbTA N)_n = M_n \tbT_{A_n} N_n \nonumber
\eeq

\section{Pre-homotopical Algebraic context}
The notion of Homotopical Algebraic context is introduced in \cite{TV4}. The reader is referred to that reference for a full definition. We call it pre-homotopical for the simple reason that we do not use $\cC_0$, nor do we need $\cA$, a sub-category of good objects, or equivalently we just work with a symmetric monoidal model category $\cC$ satisfying only the first four assumptions of \cite{TV4} which we will adapt to our bi-graded setting.\\

\subsection{$\sksMod_*$ symmetric monoidal model category}
We use the definition of symmetric monoidal model category as presented in \cite{Ho}. We already have a monoidal structure $(\bT_k, \alpha, \lambda, \rho, \Delta k)$ on $\ksMod_*$ that we prolong to a monoidal structure on $\sksMod_*$ level-wise. For that tensor product, we have an internal hom $\ucHom_{\sksMod_*}$. We have half of an adjunction of two variables:
\beq
\Hom_r = \ucHom_{\sksMod_*} \nonumber
\eeq
with:
\beq
\xymatrix{
\Hom_{\sksMod_*}(M \bT N, P) \ar[r]^-{\phi_r}_-{\cong} &\Hom_{\sksMod_*}(M, \ucHom_{\sksMod_*}(N,P))} \nonumber
\eeq
We have the braiding $\sigma: M \bT N \xrightarrow{\cong} N \bT M$ in $\ksMod_*$ that we prolonged to $\sksMod_*$. Now:
\beq
\xymatrix{
\Hom_{\sksMod_*}(M \bT N, P) \ar[d]_{\cong}^{\sigma^*} \ar@{.>}[dr]^-{\phi_l=\phi_r \circ \sigma^* } \\
\Hom_{\sksMod_*}(N \bT M, P) \ar[r]^-{\phi_r}_-{\cong} & \Hom_{\sksMod_*}(N, \ucHom_{\sksMod_*}(M,P)) } \nonumber
\eeq
\newpage
\noindent
so $(\bT, \ucHom_{\sksMod_*}, \ucHom_{\sksMod_*}, \phi_r, \phi_r \circ \sigma^*)$ is an adjunction of two variables, hence we have a closed monoidal structure on $\sksMod_*$ making it into a closed monoidal category, as defined in \cite{Ho}.\\

We now show $\bT$ is a Quillen bifunctor in $\cC = \sksMod_*$. Let $f: U \rarr V$ be a cofibration in $\cC$, $g: W \rarr X$ a cofibration in $\cC$ as well. This means they are cofibrations on the underlying bi-graded simplicial sets, so entry-wise cofibrations of simplicial sets, i.e. inclusions. More generally, maps in $\sksMod_*$ are fibrations, cofibrations or weak equivalences if and only if they are respectively fibrations, cofibrations or weak equivalences in $\skMod$, or equivalenty said, if they are so entry-wise.\\

We need:
\beq
f \Box g: (V \bT W) \coprod_{U \bT W} (U \bT X) \rarr V \oT X \nonumber
\eeq
cofibration in $\cC$, trivial if either of $f$ or $g$ is. \\

$f:U \rarr V$ and $g: W \rarr X$ being cofibrations in $\sksMod_*$ means we have entry-wise cofibrations in $\skMod$ $f_{ip}: U_{ip} \rarr V_{ip}$ and $g_{ip}: W_{ip} \rarr X_{ip}$ for $i,p=0,1$, with $U = (U_{ip})$, $V = (V_{ip})$, $W = (W_{ip})$ and $X = (X_{ip})$. Consider the coproduct $(V \bT W) \coprod_{U \bT W} (U \bT X) = \oPi (V_i \oT W_i) \coprod_{U_i \oT W_i} (U_i \oT X_i)$, which we simply denote by $\coprod$, computed in $\sksMod_*$. Levelwise, it reads:
\beq
\xymatrix{
U_i \oT W_i  \ar[d] \ar[r] & U_i \oT X_i  \ar[d]\\
V_i \oT W_i \ar[r] &\coprod_i } \nonumber
\eeq
further decomposing into parity wise coproducts:
\beq
\xymatrix{
	(U_{i0} \oT W_{i0}) \oP (U_{i1} \oT W_{i1}) \ar[d] \ar[r] & (U_{i0} \oT X_{i0}) \oP (U_{i1} \oT X_{i1}) \ar[d] \\
	(V_{i0} \oT W_{i0}) \oP (V_{i1} \oT W_{i1}) \ar[r] & \coprod_{i0}
} \nonumber
\eeq
with a similar coproduct defining $\coprod_{i1}$, in such a manner that $\coprod = (\coprod_{ip})$. Note that the maps in this diagram are induced by $f_{ip}$ and $g_{ip}$. Thus the diagram defining $\coprod_{i0}$ further decomposes into two coproduct diagrams:
\beq
\xymatrix{
	U_{ip} \oT W_{ip} \ar[d] \ar[r] & U_{ip} \oT X_{ip} \ar[d] \\
	V_{ip} \oT W_{ip} \ar[r] & C_{ipp}
} \nonumber
\eeq
so that $\coprod_{i0} = C_{i00} \oP C_{i11}$. One would obtain a similar decomposition for $\coprod_{i1}$.
Working levelwise, we want:
\beq
V_i \oT W_i \coprod_{U_i \oT W_i} U_i \oT X_i \rarr V_i \oT X_i \label{star}
\eeq
to be a cofibration. But we know $f_{ip}:U_{ip} \rarr V_{ip}$ and $g_{iq}:W_{iq} \rarr X_{iq}$ are cofibrations, so $C_{ipq} = V_{ip} \oT W_{iq} \coprod_{U_{ip} \oT W_{iq}} U_{ip} \oT X_{iq} \rarr V_{ip} \oT X_{iq}$ is a cofibration for $p,q = 0,1$, $\oT$ being a Quillen bifunctor on $\skMod$, and all such maps recombine into \eqref{star}, which is therefore a cofibration. Now if either of $f$ or $g$ is trivial, so are its entries. For instance if this is true of $f$, the entries $f_{ip}$ are trivial, making $U_{ip} \oT W_{iq} \coprod_{U_{ip} \oT W_{iq}} U_{ip} \oT X_{iq} \rarr V_{ip} \oT X_{iq}$ trivial for $i,p,q = 0,1$, and those maps recombine into $f \Box g$ trivial. This shows that $\bT$ is a Quillen bifunctor.\\

Since $k$ is the unit for $\kMod$, $\Delta k$ the unit in $\ksMod_*$, if $c_*$ is the constant simplicial functor then, $c_*( k) = k_*$ is the unit for $\skMod$ and $c_*(\Delta k ) = \Delta k_*$ is the unit for $\sksMod_*$, a constant bi-graded simplicial object. If we call $Q$ the cofibrant replacement functor in $\sksMod_*$, entrywise that can be obtained from the cofibrant replacement functor on $\skMod$. We want:
\beq
Q \Delta k_* \boxtimes X \rarr \Delta k_* \boxtimes X \label{Qkx}
\eeq
a weak equivalence for any cofibrant object $X$ in $\sksMod_*$. That means the entries of $X$ are cofibrant in $\skMod$ as well, and this latter being a symmetric monoidal category (\cite{TV}, \cite{TV4}) we have $Q k_* \oT X_{ip} \rarr k_* \oT X_{ip}$ cofibrations for $i,p=0,1$, all of which recombine into \eqref{Qkx}, a cofibration, and this for all cofibrant $X$ in $\sksMod_*$. Finally, $\bT_k$ being symmetric in $\ksMod_*$, so is its prolongation to $\sksMod_*$. Thus we have proved:
\begin{sksModstSymmMonModCat}
$\sksMod_*$ is a symmetric monoidal model category.
\end{sksModstSymmMonModCat}

\subsection{Assumption 1}
This is referred to as Assumption 1.1.0.1. in \cite{TV4}. A homotopical algebraic context has 6 assumptions, we use the first four, that is why we call our theory as being based on a pre-homotopical context. In what follows, we use the fact that $\skMod$ satisfies our four assumptions (\cite{SS}).\\
\begin{ass1}
$\sksMod_*$ is proper, pointed, and for any $X,Y \in \sksMod_*$, the natural morphism:
\beq
	QX \coprod QY \rarr X \coprod Y \rarr RX \times RY \label{QXQYRXRY}
\eeq
is an equivalence. Further $\Ho(\sksMod_*)$ is an additive category.
\end{ass1}
\begin{proof}
$\skMod$ is proper and pointed, hence so is $\sksMod_*$. Let $X, Y \in \sksMod_*$ we show the map $QX \coprod QY \rarr X \coprod Y \rarr RX \times RY$ is a natural equivalence. This composition decomposes levelwise into:
\beq
QX_i \coprod QY_i \rarr X_i \coprod Y_i \rarr RX_i \times RY_i \nonumber
\eeq
which itself decomposes as follows:
\beq
	\oP_{p,q = 0,1} QX_{ip} \coprod QY_{iq} \rarr \oP_{p,q}X_{ip} \coprod Y_{iq} \rarr \oP_{p,q} RX_{ip} \times RY_{iq} \nonumber
\eeq
thus entrywise, for $p,q$ fixed:
\beq
QX_{ip} \coprod QY_{ip} \rarr X_{ip} \coprod Y_{iq} \rarr RX_{ip} \times RY_{iq} \nonumber 
\eeq
This however is an equivalence since $\skMod$ is assumed to satisfy this assumption. All such maps recombine into \eqref{QXQYRXRY}, which is therefore an equivalence. Hence the assumption holds in $\sksMod_*$. Finally $\Ho(\sksMod_*)$ is an additive category since this is true of $\Ho(\skMod)$.
\end{proof}

\subsection{Assumption 2}
\begin{ass2}
	Let $\cC = \sksMod_*$ and $A \in \Comm(\cC) = \sksAlg_*$. Define a morphism in $\AsMod_*$ to be a fibration or an equivalence if it is so on the underlying objects of $\cC$. With this notion, $\AsMod_*$ is a combinatorial, proper model category, on which a tensor product $X \tbTA Y$ is defined as $\oP_{i,p} X_{ip} \oT_{A_{ip}}Y_{ip}$. With the monoidal structure defined by $-\tbTA-$, $\AsMod_*$ is a symmetric monoidal model category.
\end{ass2}
\begin{proof}
We first show properness, that is, left-properness and right-properness. Starting with left-properness, we must show any pushout of a weak equivalence along a cofibration is a weak equivalence. Consider any such pushout in $\AsMod_*$:
\beq
\xymatrix{
X \ar[d]_{\simeq} \ar[r] & U \ar[d]\\
Y \ar[r]_{\text{cofibr.}} & V } \nonumber
\eeq
those maps being defined on the underlying objects in $\sksMod_*$ we have that same diagram in $\sksMod_*$, which is proper, so $U \rarr V$ is a weak equivalence in $\sksMod_*$, hence so it is in $\AsMod_*$. We show right properness in like manner. \\

Regarding being combinatorial, we first show $\AsMod_*$ is cofibrantly generated. Since we will use the notations of \cite{Ho} where that concept is discussed, we remind the reader about those notations: recall that if $M$ is a model category, to say that $M$ is cofibrantly generated means there are small sets $I$ and $J$ (again not worrying about universe considerations) of morphisms in $M$, and a small cardinal $\alpha$ such that:
\begin{enumerate}
\item domains and codomains of maps in $I$ and $J$ are $\alpha$-small.
\item the class of fibrations is $J$-inj.
\item the class of trivial fibrations is $I$-inj.
\end{enumerate}
	See \cite{Ho} for the relation between $I$ and $J$. Since $\AMod$ is cofibrantly generated for $A \in \skCAlg$, it follows $A_{ip} \dashMod_{iq}$, $i,p,q=0,1$ is cofibrantly generated for $A = (A_{ip}) \in \sksAlg_*$, with small sets $I_{ipq}$ and $J_{ipq}$ such that domains and codomains of maps in those sets are $\alpha_{ipq}$-small for some $\alpha_{ipq}$. Then:
\beq
	\AsMod_*= (\oP_{p,q=0,1}A_{0p} \dashMod_{0q}) \oP (\oP_{r,s=0,1}A_{1r} \dashMod_{1s}) \nonumber
\eeq
where each summand is of the form $\AMod$ for some $A \in \skCAlg$, cofibrantly generated. Then define $I= (I_{ipq})$, $J= (J_{ipq})$, with $\alpha = \text{max}\{\alpha_{ipq} \}$, with $I$-inj the fibrations in $\AsMod_*$, $J$-inj the trivial fibrations in $\AsMod_*$, thereby cofibrantly generated. Finally recall from \cite{TV} for example, that a category $\cC$ is locally presentable if there exists a small set of $\alpha$-small objects $\cC_0 \subset \cC$ for some cardinal $\alpha$, such that any object in $\cC$ is an $\alpha$-filtered colimit of objects in $\cC_0$. It is clear a bi-graded category is locally presentable if it is so entry-wise. Then a combinatorial model category is a cofibrantly generated model category whose underlying category is locally presentable. Since for $A \in \skCAlg$, $\AMod$ is combinatorial, in particular it is locally presentable, so is $\AsMod_*$ for $A \in \sksAlg_*$, and being cofibrantly generated as well, it is therefore combinatorial.\\

We now show $\AsMod_*$ is a symmetric monoidal model category. We first prove $-\tbTA-$ gives a monoidal structure on $\AsMod_*$. $\AMod$ has an internal hom $\uHomA$ for $A \in \skCAlg$:
\beq
\Hom(M \oTA N, P) \cong \Hom(M, \uHomA(N,P)) \nonumber
\eeq
Now for $A \in \sksAlg_*$, $M,N,P \in \AsMod_*$, we have:
\begin{align}
	\Hom (M \tbTA N,P) &=\Hom(\oP_{i,p} M_{ip} \oT_{A_{ip}}N_{ip}, P_{ip}) \nonumber \\
	&=\oP_{i,p} \Hom(M_{ip} \oT_{A_{ip}}N_{ip}, P_{ip})    \nonumber \\
	&\cong \Hom(M_{ip}, \uHom_{A_{ip}}(N_{ip}, P_{ip}))  \nonumber \\
	&=\Hom(M, \uHom_A(N,P)) \nonumber
\end{align}
so there is an internal hom $\uHomA$ in $\AsMod_*$, with 
\beq
	\uHomA(N,P) = (\uHom_{A_{ip}}(N_{ip}, P_{ip})) \nonumber
\eeq
So far we have half of an adjunction of two variables:
\beq
\xymatrix{
\Hom_{\AsMod_*}(M \tbTA N, P) \ar[r]^-{\phi_r}_-{\cong} &\Hom_{\AsMod_*}(M, \uHomA(N,P)) }\nonumber
\eeq
Recall that we have the braiding $\sigma: M \bT N \rarr N \bT M$ in $\sksMod_*$ defined entry-wise, so accordingly we also have a braiding operator $\sigma_A: M \tbTA N \rarr N \tbTA M$. We have:
\beq
\xymatrix{
\Hom_{\AsMod_*}(M \tbTA N, P) \ar[d]_{\cong}^{\sigma_A^*} \ar@{.>}[dr]^{\phi_l = \phi_r \circ \sigma_A^*} \\
\Hom_{\AsMod_*}(N \tbTA M, P) \ar[r]^-{\phi_r}_-{\cong} &\Hom_{\AsMod_*}(N, \uHomA(M,P))}
\eeq
\newpage
\noindent
so $(\tbTA, \uHomA, \uHomA, \phi_r, \phi_r \circ \sigma_A^*)$ is an adjunction of two variables, hence we have a closed monoidal structure on $\AsMod_*$ making it into a closed monoidal category.\\

We now show $\tbTA$ is a Quillen bifunctor. Write $\cC = \AsMod_*$, let $f:U \rarr V$ and $\cC$, $g: W \rarr X$ be cofibrations in $\cC$. We need:
\beq
f \Box g: (V \tbTA W) \coprod_{U \tbTA W} (U \tbTA X) \rarr V \tbTA X \nonumber
\eeq
to be a cofibration in $\cC$, trivial if either of $f$ or $g$ is. This decomposes entry-wise into:
\beq
	(V_{ip} \oT_{A_{ip}} W_{ip}) \coprod_{U_{ip} \oT_{A_{ip}} W_{ip}} (U_{ip} \oT_{A_{ip}} X_{ip}) \rarr V_{ip} \oT_{A_{ip}} X_{ip} \nonumber
\eeq
for $i,p =0,1$, to be a cofibration in $\cC$, trivial if either of $f$ or $g$ is. This follows from the same computation done when showing that $\bT$ was a Quillen bifunctor on $\sksMod_*$.\\

Now $A$ being the unit in $\AsMod_*$ for $\tbTA$, $QA \rarr A$ a cofibrant replacement, for any cofibrant object $X$, we show $QA \tbTA X \rarr A \tbTA X$ is a weak equivalence. This follows from the definition of $\tbTA$ and the fact that for $A \in \skCAlg$, $X$ cofibrant in $\skMod$, $QA \oTA X \rarr A \oTA X$ is a weak equivalence. Finally, since $M \tbTA N = \oP_{i,p} M_{ip} \oT_{A_{ip}} N_{ip}$ and $- \oTA - $ defines a symmetric monoidal structure on $\AMod$, for $A \in \skCAlg$, it follows that $- \tbTA - $ defines a symmetric monoidal structure on $\AsMod_*$ as well for $A \in \sksAlg_*$. This completes the proof.
\end{proof}

\subsection{Assumption 3}
\begin{ass3}
For $A \in \Comm(\cC) = \sksAlg_*$, for any $M$ fibrant in $\AsMod_*$, the functor $-\tbTA M : \AsMod_* \rarr \AsMod_*$ preserves equivalences.
\end{ass3}
\begin{proof}
	For $M$ fibrant in $\AsMod_*$, writing $1$ for the terminal object of $\AsMod_*$, $M \rarr 1$ being a fibration, we have $M_{ip} \rarr 1_{ip}$ a fibration in $\skMod$ for $i,p=0,1$, i.e. $M_{ip}$ fibrant in $A_{ip} \dashMod$ for $i,p=0,1$. Let $f$ be an equivalence in $\AsMod_*$, $f=(f_{ip})$, $f_{ip}$ a weak equivalence in $A_{ip} \dashMod$ for $i,p=0,1$. We know for $A \in \skCAlg$, $M \in \AMod$, $-\oTA M$ preserves weak equivalences, so having $M_{ip}$ fibrant and $f_{ip}$ an equivalence, $f_{ip} \oT_{A_{ip}} M_{ip}$ is a weak equivalence for $i,p=0,1$, hence so is $f \tbTA M$ for $M \in \AsMod_*$.
\end{proof}

\subsection{Assumption 4}
\begin{ass4}
For $A \in \Comm(\cC) = \sksAlg_*$, there exist categories $\ACommC$ and $A-\Comm_{nu}(\cC)$ (non-unital) whose morphisms are fibrations and equivalences if they are so on the underlying objects of $\cC$. This makes those categories into combinatorial proper model categories. If $B$ is cofibrant in $\ACommC$, the functor $B \tbTA - : \AsMod_* \rarr \BsMod_*$ preserves equivalences.
\end{ass4}
\begin{proof}
Given how fibrations and equivalences are defined in those categories, proving they are combinatorial, proper model categories is done exactly the same way we proved those statements for $\AsMod_*$. Proving that $B \tbTA -$ preserves equivalences for $B$ cofibrant is done in the same manner that we proved $- \tbTA C$ preserves equivalences for $C$ fibrant in $\AsMod_*$.
\end{proof}

\section{Bi-graded derived algebraic stacks}
\subsection{Finitely presented morphisms}
Let $\cC = \sksMod_*$, so that $\Comm(\cC) = \sksAlg_* = M$, our model category of interest. We will need $M$ to be proper, but weak equivalences are defined on the underlying objects of $\sksMod_*$, and it being proper, so is $M$. We fix a bi-graded cofibrant resolution functor $(\Gamma: M \rarr M^{\Delta}, \iota)$, $\Gamma = (\Gamma_{ip})$ (\cite{TV}), with $\Gamma_{ip}: \uM \rarr \uM^{\Delta}$ cofibrant resolution functors in $\skMod$ for $i,p=0,1$ where as usual the underlined object corresponds to an ungraded counterpart. Here $M = \Comm(\sksMod_*)$, so $\uM = \Comm(\underline{\sksMod_*}) = \Comm(\skMod) = \skCAlg$. We have $c^* = (c^*)_{2X2}$, with weak equivalences: $\Gamma_{ip}(A) \xrightarrow{\iota_{ip}(A)} c^*A$ for $i,p=0,1$, with $c^*$ the constant cosimplicial object functor. Since we consider bi-graded morphisms in $M$, we have:
\begin{align}
\Map_M(A,B)&=\Hom_M(\Gamma(A),B) \nonumber \\
	&=\oPi \Hom_{\sksAlg}(\Gamma_i(A_i),B_i) \nonumber \\
	&= \oP_{i,p} \Hom_{\uM}(\Gamma_{ip} A_{ip}, B_{ip}) \nonumber \\
	&= \oP_{i,p} \Map_{\uM,ip}(A_{ip}, B_{ip}) \nonumber
\end{align}
so we define $\Map_{\uM} = \Map_{\uM,0} \oP \Map_{\uM,1}$ so that we have:
\beq
\Map_M = (\Map_{\uM,ip}(-,-)) = (\Hom_{\uM}(\Gamma_{ip}-,-)) \nonumber
\eeq

Now let $A \rarr B$ be a morphism in $M = \sksAlg_*$. Consider any filtered diagram of objects under $A$, $\{C_i\}_{i \in I} \in A/M$. We have a natural morphism:
\begin{align}
\hocolim_{i \in I} & \Map_{A/M}(B,C_i) \nonumber \\
	&= \Big( \hocolim_{i \in I} \Map_{A_{jp}/\uM} (B_{jp}, C_{i,jp}) \Big) \nonumber \\
	& \xrightarrow{\cong} \Big( \Map_{A_{jp}/\uM}(B_{jp}, \hocolim_{i \in I} C_{i,jp}) \Big) \nonumber \\
&=\Map_{A/M}(B, \hocolim_{i \in I}C_i) \nonumber
\end{align}
and those objects are defined in $\Ho(\Sstdot)$. This shows a morphism in $M$ is finitely presented if and only if it is so in $\uM = \skCAlg$.\\

\subsection{Bi-graded derivations and cotangent complexes}
For $\cC = \sksMod_*$, $A \in \Comm(\cC)$, $M \in$ $\ACommC$, we define a commutative monoid $A \oP M$ with $A \coprod M$ as underlying object, just as in \cite{TV4}, with multiplication:
\beq
(A \oP M) \bT (A \oP M) \xrightarrow{\pi} A \oP M \nonumber
\eeq
defined by:
\begin{align}
(A \coprod M) &\bT (A \coprod M) \nonumber \\
&=A \bT A \coprod A \bT M \coprod M \bT A \coprod M \bT M \nonumber \\
&\longrightarrow A \coprod M \nonumber
\end{align}
where:
\begin{align}
\mu \coprod *: &A \bT A \xrightarrow{\mu} A \rarr A \coprod M \nonumber\\
* \coprod \rho: &A \bT M \xrightarrow{\rho} M \rarr A \coprod M \nonumber \\
&M \bT M \rarr M \rarr A \coprod M \nonumber
\end{align}
Here $\mu$ and $\rho$ have entries as defined in \cite{TV4} and where in the last line we use the fact that $M$ is a commutative monoid in $\AsMod_*$. To show $A \oP M$ is commutative, it suffices to write:
\beq
A \coprod M = \oPi A_i \coprod M_i = \oPi \oP_p A_{ip} \coprod M_{ip} \nonumber
\eeq
Then:
\begin{align}
	(A \coprod M) \bT (A \coprod M) &= \oPi (A_i \coprod M_i) \oT (A_i \coprod M_i) \nonumber \\
	&= \oPi (\oP_p A_{ip} \coprod M_{ip}) \oT (\oP_q A_{iq} \coprod M_{iq}) \nonumber \\
	&= \oPi \oP_{p,q}(A_{ip} \coprod M_{ip}) \oT (A_{iq} \coprod M_{iq}) \nonumber
\end{align}
Now $A \in \ksMod_*$ is super-commutative if $\mu \circ \sigma = \mu$, that is for any homogeneous elements $x,y$ in $A$, $xy = \mu( x \oT y) = (-1)^{|x||y|} yx$. Presently we have $\pi$ on:
\begin{align}
	(A \oP M) \bT (A \oP M) &= \oPi \oP_{p,q}(A_{ip} \oT A_{iq}) \coprod (A_{ip} \oT M_{iq}) \nonumber \\
	&\coprod (M_{ip} \oT A_{iq}) \coprod (M_{ip} \oT M_{iq}) \nonumber
\end{align}
and on each entry $\pi$ is super-commutative, so $A \oP M$ is super-commutative.\\

Now for $A \rarr B$ a morphism in $\Comm(\cC)$, $M \in$ $\BCommC$, the morphism $id \coprod *: B \oP M \rarr B$ is a morphism in $\ACommC$, hence we can regard $B \oP M \in \ACommC/B$.
\begin{der}
For a morphism $A \rarr B$ in $\Comm(\cC)$, $M \in \BCommC$, $\cC = \sksMod_*$, we define, as in \cite{TV4}, the set of $A$-derivations $B \rarr M$ as:
\beq
\DerA(B,M) = \Map_{\ACommC/B}(B, B \oP M) \in \Ho(\Sstdot) \nonumber
\eeq
where $B \rarr B \oP M$ is a section of $id \coprod *$.
\end{der}
We have the following result, much as in \cite{TV4}:
\begin{MapDer}
	For any morphism $A \rarr B$ in $\sksAlg_*$, there is an object $\LBA \in \BsMod_*$, there is a $d \in \pi_0(\DerA(B, \LBA))$ such that for any $M = (M_{ip})$, $M_{ip} \in B_{ip}\dashMod$, i.e. $ M \in \BsMod_*$, the natural induced morphism:
\beq
d^*: \Map_{\BsMod_*}(\LBA,M) \rarr \DerA(B,M) \nonumber
\eeq
is an isomorphism in $\Ho(\Sstdot)$.
\end{MapDer}
Note that for $\psi \in \Map_{\BsMod_*}(\LBA, M)$, the element $d^* \psi$ is constructed as follows:
\beq
\xymatrix{
&\Gamma \LBA \ar[d] \ar[r]^{\psi} &M \ar[d]\\
B \ar[r]^-d & B\oP  \Gamma \LBA \ar[r]^{B \oP \psi} &B \oP M } \nonumber
\eeq
\begin{proof}
A morphism $A \rarr B \in \Comm(\cC)$ gives rise to entry maps $A_{ip} \rarr B_{ip} \in \Comm(\ucC)$, with $\ucC = \skMod$, for $i,p=0,1$, so we are guaranteed by \cite{TV4} that there is some $\LBipAip \in B_{ip}\dashMod$ and a $d_{ip} \in \pi_0(\Der_{A_{ip}}(B_{ip}, \LBipAip))$ such that for any $M_{ip} \in B_{ip}\dashMod$:
\beq
	d_{ip}^*: \Map_{B_{ip}\dashMod}(\LBipAip, M_{ip}) \rarr \Der_{A_{ip}}(B_{ip}, M_{ip}) \nonumber
\eeq
is an isomorphism in $\Ho(\SetD)$. We view $\LBA = (\LBipAip)$ as an object of $(B_{ip}\dashMod_{ip}) \subset \BsMod_*$, and the same is true of $M = (M_{ip})$. Then it suffices to write:
\begin{align}
	\Map_{\BsMod_*}(\LBA,M) &= \oP_{i,p} \Map_{B_{ip}\dashMod}(\LBipAip, M_{ip}) \nonumber \\
	&\xrightarrow{(d^*_{ip})} \oP_{i,p} \Map_{A_{ip}-\Comm(\cC)/B_{ip}}(B_{ip}, B_{ip} \oP M_{ip}) \nonumber \\
&=\Map_{\ACommC/B}(B, B \oP M) \nonumber \\
&=\DerA(B,M) \nonumber
\end{align}
with $d = (d_{ip}) \in \pi_0(\DerA(B,\LBA))$.
\end{proof}
\begin{LBAdef}
Let $A \rarr B$ a morphism in $\Comm(\sksMod_*)$. The $B$-module $\LBA \in \Ho(\BsMod_*)$ is called as in \cite{TV4} the cotangent complex of $B$ over $A$. If $A = 1 = \Delta k_*$, write $\LB = \mathbb{L}_{B/1}$, which we will call the cotangent complex of $B$.
\end{LBAdef}

\subsection{\'{e}tale morphisms}
We have shown that $\AsMod_*$ is a symmetric monoidal model category with the tensor product $\tbTA$. This means its homotopy category $\Ho(\AsMod_*)$ has a natural symmetric monoidal structure $\tbTAL$.
\begin{formetale}
Let $f:A \rarr B$ be a morphism in $\Comm(\sksMod_*) = \sksAlg_*$. The morphism $f$ is said to be formally \'{e}tale if:
\beq
\LA \tbTA B \xrightarrow{\cong} \LB \nonumber
\eeq
or entry-wise:
\beq
	\bL_{A_{ip}} \oT_{A_{ip}} B_{ip} \xrightarrow{\cong} \bL_{B_{ip}} \nonumber
\eeq
for $i,p=0,1$, i.e. $f$ is formally \'{e}tale if it is so entry-wise.
\end{formetale}
Note that this is a well-defined morphism: if $p = 0$, the parity of $\bL_{A_{ip}} \oT_{A_{ip}} B_{ip}$ is $p + p + p = 0$, so we do have $\bL_{A_{i0}} \oT_{A_{i0}}B_{i0} \rarr \bL_{B_{i0}}$, and if $p = 1$, the parity of $\bL_{A_{i1}} \oT_{A_{i1}}B_{i1} $ is 3 mod 2, which is 1, so we also have a well-defined map  $\bL_{A_{i1}} \oT_{A_{i1}}B_{i1} \rarr \bL_{B_{i1}}$.\\

\begin{etale}
We define a morphism of commutative monoids in $\sksMod_*$ to be \'{e}tale if it is finitely presented and formally \'{e}tale.
\end{etale}
This is the same definition as initially presented in \cite{TV4}. Following the same reference, we will use the same terminology for morphisms in $\Ho(\CommC)$ and for the corresponding morphisms of representable stacks. We typically write $\AffC = (\CommC)^{\op}$ but if $\cC = \sksMod_*$, $\CommC = \sksAlg_*$, $(\sksAlg_*)^{\op}$ is then denoted, much as in \cite{TV4}, $\kDsAff_*$.\\

If $\cC = \sksAlg_* \subset \sksMod_*$, $A \in \cC$, we define the underlying bi-graded space of $A$ as:
\begin{align}
|A| &= \MapC(\Delta k_*, A) \nonumber \\
	&=(\Map_{\ucC}(k_*, A_{ip})  \nonumber \\
	&= (|A_{ip}|) \in \Sstdot \nonumber
\end{align}
so that following \cite{TV4}, we define:
\begin{align}
\pi_k(A) &= \pi_k(|A|, *) \nonumber \\
	&=(\pi_k(|A_{ip}|, *_{ip})) \nonumber \\
	&= (\pi_k A_{ip})\nonumber
\end{align}

Write $\Spec A \in \kDsAff_*$ if $A \in \sksAlg_*$. Following \cite{TV4}, we define families $\{ \Spec A_i \rarr \Spec A \}_{i \in I}$ to be \'{e}tale covering families if and only if there is some $J \subset I$ such that for all $i \in I$:
\beq
\pi_n A \tilde{\bT}_{\pi_0 A} \pi_0 A_i \xrightarrow{\cong} \pi_n A_i \label{pinAbTarr}
\eeq
and
\beq
\coprod_{j \in J} \Spec \pi_0 A_j \rarr \Spec \pi_0 A \nonumber
\eeq
is \'{e}tale and surjective, where $\Spec A \rarr \Spec B$ is \'etale if $B \rarr A$ is an \'etale morphism in $\sksAlg_*$. 
\newpage

\noindent
Note that \eqref{pinAbTarr} reads as follows entry-wise:
\beq
\pi_n A_{jp} \oT_{\pi_0 A_{jp}} \pi_0 A_{i,jp} \xrarr{\cong} \pi_n A_{i,jp} \nonumber
\eeq
Now if $A \in \ksAlg_*$:
\begin{align}
	\Spec A &= \Spec \oP_{i,p} A_{ip} \nonumber \\
	&= (\oP_{i,p} A_{ip})^{\op} \nonumber \\
	&= \oP_{i,p} A_{ip}^{\op} \nonumber \\
	&= \oP_{i,p} \Spec A_{ip} \nonumber
\end{align}
along with $\pi_0 A = (\pi_0 A_{ip})$, this gives us, for $A \in \sksAlg_*$:
\beq
\Spec \pi_0 A = \Spec \oP_{i,p} \pi_0 A_{ip} = \oP_{i,p} \Spec \pi_0 A_{ip} \nonumber
\eeq
It follows:
\begin{align}
	\coprod_{j \in J} \Spec \pi_0 A_j & = \coprod_{j \in J} \oP_{l,p} \Spec \pi_0 A_{j,lp} \nonumber \\
	&= \oP_{l,p} \coprod_{j \in J} \Spec \pi_0 A_{j,lp} \surjfleche^{\tet} \oP_{l,p} \Spec \pi_0 A_{lp} \nonumber 
\end{align}
so that entrywise:
\beq
\coprod_{j \in J} \Spec \pi_0 A_{j,lp} \surj \Spec \pi_0 A_{lp} \nonumber \\ 
\eeq
However, finitely presented morphisms in $\sksAlg_*$ are so entry-wise, and given the definition of formally \'etale morphisms and the definition of $\tbTA$, formally \'etale maps are so entry-wise as we argued. It follows \'etale morphisms in $\sksMod_*$ are so entry-wise, so the above morphism is entry-wise \'etale. We conclude that a family of coverings in $\kDsAff_*$ is an \'etale covering if and only if it is so entry-wise in $\kDAff$.\\

\subsection{Descent on super modules}
We now define the descent condition on supermodules, which is just an adaption to our setting of the work done in \cite{TV4}. We define a cosimplicial object $A_*$ in $\CommC$, $\cC = \ksMod_*$, to be an object of $\CommC^{\Delta}$, which to $[n]$ associates $A_n = (A_{n,ip})$. A cosimplicial $A_*$-module $M_*$ is given by a graded $A_n$-supermodule $M_n$ for all $n \in \Delta$, $M_n = (M_{n,ip})$, and for any morphism $u:[n]\rarr [m]$ in $\Delta$ of a morphism of graded $A_n$-supermodules $\alpha_u: M_n \rarr M_m$ satisfying the usual covariance condition. In the same manner, a morphism of cosimplicial graded $A_*$-supermodules $f:M_* \rarr N_*$ is given by a collection of morphisms $f_n:M_n \rarr N_n$ for all $n\in \Delta$ commuting with the $\alpha$'s. This defines a category $\csAstsMod_*$ of co-simplicial graded $A_*$-supermodules. This is a combinatorial, proper model category where equivalences (resp. fibrations) are morphisms $f: M_* \rarr N_*$ such that each $f_n: M_n \rarr N_n$ is an equivalence (resp. a fibration) in $A_n$-sMod$_*$, a level-wise projective model structure. It will also be convenient sometimes, for $A \in \ksAlg_*$, to regard $B_*$ a co-simplicial commutative graded $A$-superalgebra as a co-simplicial commutative monoid with a co-augmentation morphism $A \rarr B_*$, where we regard $A$ as a constant co-simplicial object. This latter defines a category $\csAsMod_*$ of co-simplicial graded cc($A$)-supermodules, which is really cs($\AsMod_*$), along with its levelwise projective model structure.\\

For a co-simplicial graded $A$-super module $M_*$, we define a co-simplicial graded $B_*$-supermodule $B_* \tbTA M_*$ by $(B_* \tbTA M_*)_n = B_n \tbTA M_n$, with morphisms between different degrees given by the ones on $B_*$ and $M_*$. This gives a functor:
\beq
B_* \tbTA -: \csAsMod_* \rarr \csBstsMod_* \nonumber
\eeq
with a right adjoint:
\beq
\csAsMod_* \leftarrow \csBstsMod_*: f \nonumber
\eeq
sending a graded $B_*$-super module $M_*$ to its underlying graded $A$-super module, clearly a Quillen adjunction. We have an additional adjunction, as in \cite{TV4}:
\beq
\cst : \Ho(\AsMod_*) \rightleftarrows \Ho(\csAsMod_*): \holim = \mathbb{R} \lim \nonumber
\eeq
and by composition:
\beq
B_* \tbTAL - = \mathbb{L}(B_* \tbTA -) \circ \cst: \Ho(\AsMod_*) \rightleftarrows \Ho(\csBstsMod): \holim \circ \bR f = \int \nonumber
\eeq
\begin{htpycart}
Let $B_*$ be a co-simplicial commutative graded super monoid, $M_*$ a co-simplicial graded $B_*$ super module. We say $M_*$ is homotopy cartesian if for all $ u :[n] \rarr [m]$ in $\Delta$, the map induced by $\alpha_u: M_n \rarr M_m$, $M_n \bT_{B_n}^{\mathbb{L}} B_m \rarr M_m$ is an isomorphism in $\Ho(B_m$-sMod$_*$).
\end{htpycart}
Now if $A \in \CommC$, $\cC = \ksMod_*$, $B_*$ co-simplicial commutative algebra over $A$, seen as a co-simplicial commutative graded super monoid, we say the augmentation $A \rarr B_*$ satisfies the descent condition if in the adjunction:
\beq
B_* \tbTAL - : \Ho(\AsMod_*) \rightleftarrows \Ho(\csBstsMod_*): \int \nonumber
\eeq
$B_* \tbTAL -$ is fully faithful and gives an equivalence between $\Ho(\AsMod_*)$ and the full subcategory of $\Ho(\csBstsMod_*)$ spanned by homotopy cartesian objects.\\

\subsection{\'{e}tale model topology}
In this subsection we prove that \'{e}tale covering families define a model topology on $\kDsAff_*$, which further satisfies an assumption that we will refer to as the "Cover" assumption, the bi-graded generalization of Assumption 1.3.2.2 of \cite{TV4}.\\

\subsubsection{Model topology}
Recall that a family of morphisms $\{ \Spec A_i \rarr \Spec A \}_{i \in I}$ in $\kDsAff_*$ is an \'{e}tale covering family if and only if there is a finite subset $J \subset I$ such that for all $i \in I$, $\pi_* A \tilde{\bT}_{\pi_0 A} \pi_0 A_i \xrightarrow{\cong} \pi_* A_i$ and $\coprod_{j \in J} \Spec \pi_0 A_j \rarr \Spec \pi_0 A$ is \'{e}tale and surjective. Recall that being \'etale means finitely presented and formally \'etale, where $A \rarr B$ is formally \'etale if $\bL_A \tbTA B \xrarr{\cong} \bL_B$, and finitely presented means for any familiy $\{C_i \}$, $\hocolim_{i \in I} \Map_{A/M}(B, C_i) \xrarr{\cong} \Map_{A/M}(B, \hocolim C_i)$. Recall also that a family is a covering \'etale family if and only if it is so entry-wise. We will use the fact, proven in \cite{TV4}, that \'{e}tale covering families define a model topology in $\kDAff = (\skCAlg)^{\op}$. We refer the reader to \cite{TV} for the notion of model topology. There are three conditions to be met. Let $M = \kDsAff_*$, let $\Spec A \in M$. Suppose we have an isomorphism $\Spec B \rarr \Spec A$. Entry-wise, this reads $\Spec B_{ip} \rarr \Spec A_{ip}$ isomorphism for $i,p=0,1$ in $\kDAff$, which gives $\{\Spec B_{ip} \rarr \Spec A_{ip} \} \in \covet(\Spec A_{ip})$ for $i,p=0,1$. But since $\covet(X) = (\covet(X_{ip}))$, it follows that condition 1 is met. For the second point, let $\{ \Spec A_i \rarr \Spec A \}_{i \in I} \in \covet(\Spec A)$. This means $\{ \Spec A_{i,kp} \rarr \Spec A_{kp} \}_{i \in I} \in \covet(\Spec A_{kp})$ for $k,p=0,1$. Let $\{ \Spec B_{ij} \rarr \Spec A_i \}_{j \in J} \in \covet(\Spec A_i)$ for $i \in I$, and again that means $\{ \Spec B_{ij,kp} \rarr \Spec A_{i,kp} \}_{j \in J} \in \covet(\Spec A_{i,kp})$ for $k,p =0,1$. Then $\{ \Spec B_{ij,kp} \rarr \Spec A_{kp} \}_{i \in I, j \in J} \in \covet(\Spec A_{kp})$ for $k,p=0,1$, i.e. $\{ \Spec B_{ij} \rarr \Spec A \}_{i \in I, j \in J} \in \covet(\Spec A)$. This is the second condition. Third point, for any $\{ \Spec A_i \rarr \Spec A \}_{i \in I} \in \covet(\Spec A)$, that is $\{ \Spec A_{i,jp} \rarr \Spec A_{jp} \}_{i \in I} \in \covet(\Spec A_{jp})$ for $j,p=0,1$, for any morphism $\Spec B \rarr \Spec A$, which entry-wise reads $\Spec B_{jp} \rarr \Spec A_{jp}$, then $\{ \Spec A_{i,jp} \times_{\Spec A_{jp}}^h \Spec B_{jp} \rarr \Spec B_{jp} \} \in \covet(\Spec B_{jp})$ for $j,p=0,1$, which recombine into $\{ \Spec A_i \times_{\Spec A}^h \Spec B \rarr \Spec B \} \in \covet(\Spec B)$. Hence \'{e}tale covering families define a model topology on $\kDsAff_*$.\\

Now part of showing the cover assumption is satisfied hinges on the fact that $(\sAffC)_*$ is a simplicial model category on which we can put a Reedy model structure.

\subsubsection{ Reedy model structure on $\sAffC$ }
$\sAffC = \AffC^{\Dop}$ with $\AffC = \kDsAff_* = (\sksAlg_*)^{\op}$, is a simplicial category with $\cC = \sksMod_*$, thus it is cotensored over $\SetD$ so for $X_* \in \sAffC$, $K \in \SetD$, $X_*^K$ is well defined, and its zeroth part will just be denoted $X^K$. We have in particular:
\beq
X_*^K = (\oP_{i,p}X_{*,ip})^K = \oP_{i,p}X_{*,ip}^K  \nonumber
\eeq
Indeed, for $X_*,Y_* \in \sAffC$ and $K \in \SetD$:
\begin{align}
	\Hom(Y_* \oT K, X_*) &= \Hom((\oP_{i,p}Y_{*,ip}) \oT K, (\oP_{i,p}X_{*,ip})) \nonumber \\
	&= \oP_{i,p} \Hom_{\sAffuC}(Y_{*,ip} \oT K, X_{*,ip}) \nonumber \\
	&\cong \oP_{i,p} \Hom(Y_{*,ip}, X_{*,ip}^K) \nonumber 
\end{align}
since $\sAffuC$ is a simplicial category. This further recombines into:
\beq
\Hom(Y_*, \oP_{i,p}X_{*,ip}^K) = \Hom(Y_*,X_*^K) \nonumber 
\eeq
We also have:
\begin{align}
	X^K &= (\oP_{i,p} X_{*,ip}^K)_0 \nonumber \\
	&=\oP_{i,p}(X_{*,ip}^K)_0 \nonumber \\
	&= \oP_{i,p} X_{ip}^K \nonumber \\
	&= (X_{ip}^K) \nonumber
\end{align}
Thus the simplicial structure on $\sAffC$ is an entry-wise generalization of that on $\sAffuC$. One could also show that it is a bi-graded simplicial category and this is proved formally in exactly the same manner that we showed $\sksMod_*$ is a bi-graded simplicial category. Then for the bi-graded simplicial structure, we have $X_*^K = \oP_{i,p,q} K_{*,ip}^{K_{iq}}$. If $K = \Dn \times \Dn$, with $\Dn$ regarded as being of even parity, we then have:
\beq
K_*^{\Delta \Delta^n} = \oP_{i,p} K_{*,ip}^{\Delta^n} = (K_*)^{\Dn} \nonumber
\eeq
for the simplicial structure. For $Y \in \AffC \hookrightarrow \sAffC$ regarded as a constant simplicial object via the constant simplicial functor cs$_*(Y)$, with cs$_*(Y)_n = Y$ for any $n$, when we write $Y^K$ we really mean $(\text{cs}_*(Y))^K$.\\

We now put a Reedy structure on $\sAffC$. Equivalences  $X_* \rarr Y_*$ are such that $X_n \rarr Y_n$ are equivalences in $\AffC$ for any $n$, and fibrations $X_* \rarr Y_*$ are those maps such that for any $n$:
\beq
\xymatrix{
X^{\Dn} \cong X_n \ar@{=}[d] \ar[r] & X^{\pDn} \times_{Y^{\pDn}} Y^{\Dn} \ar@{=}[d] \nonumber\\
(X_{ip}^{\Dn}) & (X_{ip}^{\pDn} \times_{Y_{ip}^{\pDn}} Y_{ip}^{\Dn})   \nonumber 
}
\eeq
is a fibration in $\AffC$. Entry-wise this decomposes into $X_{ip}^{\Dn} \rarr X_{ip}^{\pDn}\times_{Y_{ip}^{\pDn}}Y_{ip}^{\Dn}$, fibration in $\AffuC$ for $i,p=0,1$, i.e. the Reedy model structure on $\sAffC$ is an entry-wise Reedy model structure on $\sAffuC$. Observe that for $K \in \SetD$, the functor:
\begin{align}
\sAffC & \rarr \AffC \nonumber \\
X_* & \mapsto (X_*^K)_0 = X^K \nonumber
\end{align}
is a right Quillen functor for the Reedy model structure on $\sAffC$. 

\newpage
Indeed it decomposes into $X_* = (X_{*,ip}) \mapsto (X_{ip}^K)$, right Quillen in the Reedy model structure on $\sAffuC$ entry-wise, recombining into a right Quillen functor in $\sAffC$. Its right derived functor is given by:
\begin{align}
\Ho(\sAffC) &\rarr \Ho(\AffC) \nonumber \\
X_* & \mapsto \mathbb{R}X^k \nonumber \\
	&= \bR(X_{ip}^K) \nonumber \\
	&= (\bR X_{ip}^K)\nonumber \\
	&= (X_{ip}^{\bR K})\nonumber \\
&=X^{\mathbb{R}K} \nonumber
\end{align}

\subsubsection{ The Cover Assumption}
This assumption, which mirrors Assumption 1.3.2.2 of \cite{TV4}, is just the bi-graded version of the latter, so the reader is referred to that reference for the full statement. Essentially we assume the \'{e}tale topology on $\Ho(\kDsAff_*)$ is quasi-compact,
projections onto coproducts form \'{e}tale-covering families, both of which are fairly immediate to verify, and thirdly if $X_* \rarr Y$ is an augmented simplicial object in $\AffC$ corresponding to a co-augmented co-simplicial object $A \rarr B_*$ in $\CommC$, if we assume that for all $n$ the morphism:
\beq
X_n \rarr X^{\bR \pDn} \times_{Y^{\bR \pDn}}^h Y \label{etcovmap}
\eeq
by itself forms an \'{e}tale covering family in $\AffC$, then $A \rarr B_*$ satisfies descent. We discuss this third point.\\

We have:
\beq
 X^{\bR \pDn} \times_{Y^{\bR \pDn}}^h Y = (X_{ip}^{\bR \pDn} \times^h_{Y_{ip}^{\bR \pDn}} Y_{ip}) \nonumber
\eeq
with this the map \eqref{etcovmap} reads entry-wise:
\beq
X_{n,ip} \rarr  X_{ip}^{\RpDn} \times_{Y_{ip}^{\RpDn}}^h Y_{ip} \nonumber
\eeq
for $i,p=0,1$. Now since \'{e}tale covering families are so entry-wise, those maps correspond to \'etale covering families in $\AffuC$, so the maps $A_{ip} \rarr B_{*,ip}$ satisfy descent by the original ungraded cover assumption of \cite{TV4}. 

\newpage
We show this implies they can be recombined into $A \rarr B_*$ that also satisfies descent. This will imply our bi-graded cover assumption is just a diagonalization of the original cover assumption of \cite{TV4}.\\

Recall that to satisfy descent means
\beq
B_* \tbTAL -: \Ho(\AsMod_*) \rarr \Ho(\csBstsMod_*) \nonumber 
\eeq
is fully faithful and establishes an equivalence of $\Ho(\AsMod_*)$ with the subcategory of homotopy cartesian objects in $\Ho(\csBstsMod_*)$. That the entries $A_{ip} \rarr B_{*,ip}$ satisfy descent then means that  $B_{*,ip} \oT_{A_{ip}}^{\mathbb{L}}-: \Ho(A_{ip} \dashMod) \xrightarrow{ff} \Ho(\text{cs}B_{*,ip}\text{-Mod})$ for $i,p=0,1$. Those maps recombine into $B_* \tbTAL - : \Ho(\AsMod_*) \rarr \Ho(\csBstsMod_*)$ fully faithful. Regarding homotopy cartesian objects, those are objects $M_*$ in $\Ho(\csBstMod_*)$ such that $M_n \tbT^{\bL}_{B_n} B_m \rarr M_m$ is an isomorphism in $\Ho(B_m\text{-sMod}_*)$. We quickly check $B_* \tbTAL -$ maps objects of $\Ho(\AsMod_*)$ into cartesian objects of $\Ho(\csBstsMod_*)$: let $C_* \in \Ho(\AsMod_*)$. We know from \cite{TV4} in the ungraded case that entry-wise:
\beq
(B_{n,ip} \oT_{A_{ip}}^{\mathbb{L}} C_{n,ip}) \oT_{B_{n,ip}}^{\mathbb{L}} B_{m,ip} \cong B_{m,ip} \oT_{A_{ip}}^{\mathbb{L}} C_{m,ip} \nonumber
\eeq
for $i,p=0,1$, since $B_* \oTAL -$ maps objects of $\Ho(\AMod)$ into homotopy cartesian objects of $\Ho(\csBstMod)$. Now we have:
\begin{align}
	(B_n \tbTAL C_n) \tbT^{\bL}_{B_n} B_m &= \oP_{i,p} (B_n \tbTAL C_n)_{ip} \oT_{B_{n,ip}}^{\mathbb{L}} B_{m,ip} \nonumber \\
	& = \oP_{i,p} (B_{n,ip} \oT_{A_{ip}}^{\mathbb{L}} C_{n,ip}) \oT_{B_{n,ip}}^{\mathbb{L}} B_{m,ip} \nonumber \\
	&\cong \oP_{i,p} B_{m,ip} \oT_{A_{ip}}^{\mathbb{L}} C_{m,ip} = B_m \tbTAL C_m \nonumber
\end{align}
hence $B_* \tbTAL -$ maps objects of $\Ho(\AsMod_*)$ into homotopy cartesian objects. The essential surjectivity is dealt with entry-wise. Thus the third point is satisfied.\\

To conclude this subsection on the model topology, the \'{e}tale topology does define a model topology on $\kDsAff_*$ and makes $(\kDsAff_*, \acute{e}t.)$ into a model site. Further this topology satisfies our technical cover assumption, bi-graded generalization of the cover assumption as originally introduced in \cite{TV4}.\\

\subsection{Bi-graded simplicial presheaves}
Let $M$ be a small bi-graded model category (the choice of universes $\mathbb{U} \in \mathbb{V}$ being implied, $M$ $\mathbb{V}$-small, $\mathbb{U}$-cofibrantly generated...), $W(M)$ its class of equivalences. For us $M = \kDsAff_*$. Define:
\beq
\sPrZtwosqu(M) = \Sstdot^{\Mop} \nonumber
\eeq
the category of bi-graded simplicial presheaves on $M = (M_{ip})$. Equivalently:
\beq
\sPrZtwosqu(M) = \oP_{i,p} \sPr(\uM) \nonumber
\eeq
We put the entry-wise projective model structure on $\sPrZtwosqu(M)$ where equivalences and fibrations are defined object-wise. \\

Our aim at this point is to define a model category of prestacks $\Mhat$, obtained as a bi-graded left Bousfield localization of $\sPrZtwosqu(M)$ along $\{h_u \, | \, u \in W(M) \}$ where $h:M \rarr \PrZtwo(M) \hookrightarrow \sPrZtwosqu(M)$ is the bi-graded Yoneda embedding. In other terms we localize each individual $\sPr(\uM)$, i.e. we do an entry-wise Bousfield localization. Then $\Ho(\Mhat)$ is identified with the subcategory of $\Ho(\sPrZtwosqu(M))$ that consists of presheaves that preserve weak equivalences entry-wise. $\Mhat$ is a bi-graded simplicial model category since $\sPr(\uM)^{\wedge} = \sPr(\kDAff)^{\wedge}$ is a simplicial model category from \cite{TV}, and we denote by $\RuHom \in \Ho(\Sstdot)$ its derived hom. From there we define a model category $\Mtet$ of stacks on $(M, \acute{e}t.)$, left bi-graded Bousfield localization of $\Mhat$ along homotopy \'{e}tale-hypercovers. We obtain a pair of adjoint Quillen functors:
\beq
id: \Mhat \rightleftarrows \Mtet: id \nonumber
\eeq
giving rise to:
\beq
a = \mathbb{L}id: \Ho(\Mhat) \rightleftarrows \Ho(\Mtet): \mathbb{R}id = j \nonumber
\eeq
we finally define a bi-graded stack on $(\kDsAff_*, \tet)$ to be an object of $\sPrZtwosqu(M)$ whose image in $\Ho(\Mhat)$ is in the essential image of the functor $j$, exactly in the same way that derived stacks were defined in \cite{TV4}.\\

\newpage
\subsection{Bi-graded left Bousfield localization}
The reference we use for Bousfield localizations is \cite{Hi}, which we follow closely since it is just a matter of adapting Hirschhorn's definition to the bi-graded case. Let $M = (M_{ip})$ be a $\bZ_2$-graded model category, such as $\kDsAff_*$. Let $\cC$ be a class of maps in $M$, $\cC = (\cC_{ip})$. An object $W$ of $M$ is $\cC$-local if it is fibrant and for any $f: A \rarr B$ in $\cC$, or entry-wise $f_{ip}: A_{ip} \rarr B_{ip}$ in $\ucC$ for $i,p=0,1$, the induced map of mapping spaces is a weak equivalence:
\beq
f^* : \Map(B,W) \xrightarrow{\simeq} \Map(A,W) \nonumber
\eeq
or entry-wise:
\beq
f_{ip}^*: \Map(B_{ip},W_{ip}) \xrightarrow{\simeq} \Map(A_{ip},W_{ip}) \nonumber
\eeq
In other terms, $\cC$-local objects are entry-wise $\ucC$-local objects. We define a map $g: X \rarr Y$ in $M$ to be a $\cC$-local equivalence if for any $\cC$-local object $W$, the induced map of mapping spaces:
\beq
g^*: \Map(Y,W) \rarr \Map(X,W) \nonumber
\eeq
is a weak equivalence, or equivalently entry-wise:
\beq
g_{ip}^*: \Map(Y_{ip},W_{ip}) \xrightarrow{\simeq} \Map(X_{ip},W_{ip}) \nonumber
\eeq
for $i,p=0,1$, so again $\cC$-local equivalences are entry-wise $\ucC$-local equivalences. \\

Finally we define the bi-graded left Bousfield localization of $M$ with respect to $\cC$ to be a model category structure $\LCM$ on the underlying bi-graded category of $M$ such that the class of weak equivalences is the class of $\cC$-local equivalences, cofibrations are those of $M$, and fibrations have the right lifting property with respect to those cofibrations that are also $\cC$-local equivalences. If $A \rarr B$ is such a cofibration, $X \rarr Y$ a fibration, a diagram such as:
\beq
\xymatrix{
A \ar[d] \ar[r] & X \ar[d]\\
B \ar@{.>}[ur] \ar[r] &Y} \nonumber
\eeq
breaks up into diagrams for $i,p=0,1$:
\beq
\xymatrix{
A_{ip} \ar[d] \ar[r] & X_{ip} \ar[d]\\
B_{ip} \ar@{.>}[ur] \ar[r] &Y_{ip}} \nonumber
\eeq
with $A_{ip} \rarr B_{ip}$ a cofibration that's also a $\ucC$-local equivalence. This shows fibrations are also entry-wise fibrations, hence:
\beq
\LCM = (L_{\cC_{ip}} M_{ip})\nonumber
\eeq

\subsection{Model category of prestacks $\Mhat$}
We first define restricted diagrams, following \cite{TV}. For $M$ a small category (relative to $\mathbb{V}$), such as $\kDsAff_*$, given $\Sstdot$, bi-graded simplicial model category, cofibrantly generated, we consider $\Sstdot^M$ the category of bi-graded simplicial functors $M \rarr \Sstdot$ with its entry-wise projective model structure. For any $x \in M$ we have an induced map:
\begin{align}
\iota_x^*: & \Sstdot^M \rarr \Sstdot \nonumber \\
&F \mapsto F(x) \nonumber
\end{align}
with a left adjoint $(\iota_x)_!: \Sstdot \rarr \Sstdot^M$ that's a left Quillen functor. Let $I$ be a set of generating cofibrations in $\Sstdot$, $f: A \rarr B$ any morphism in $I$, $u: x \rarr y$ any morphism in $M$. Consider the natural morphism in $\Sstdot^M$:
\beq
f \Box u: (\iota_y)_! A\coprod_{(\iota_x)_! A} (\iota_x)_! B \rarr (\iota_y)_! B \nonumber
\eeq
with $(\iota_x)_{!jp}: S_{jp} \rarr S_{jp}^{M_{jp}}$, left adjoint to
\beq
(\iota_x^*)_{jp}: F_{jp} \mapsto (F(x))_{jp} = F_{jp}(x_{jp}) \nonumber
\eeq
that is, $(\iota_x^*)_{jp} = \iota_{x_{jp}}^*$, so that $(\iota_x)_{!jp} = (\iota_{x_{jp}})_!$. Thus:
\beq
(\iota_y)_!A = ((\iota_{y_{jp}})_{!} A_{jp}) \nonumber
\eeq
which means $f \Box u = (f_{jp} \Box u_{jp})$, so that entry-wise:
\beq
f_{lp} \Box u_{lp}: (\iota_{y_{lp}})_!A_{lp} \coprod_{(\iota_{x_{lp}})_!A_{lp}}(\iota_{x_{lp}})_!B_{lp} \rarr (\iota_{y_{lp}})_!B_{lp} \nonumber
\eeq
\begin{resdiagr}
The model category of restricted diagrams from $M$ to $\Sstdot$, denoted $M^{\op,\wedge}$, is defined to be the left Bousfield localization of $\Sstdot^M$ along the set of morphisms of the form $f \Box u$, for $f \in I$, and $u$ a weak equivalence in $M$.
\end{resdiagr}
If we denote by $W$ the set of weak equivalences, this gives:
\beq
L_{I \Box W} \Sstdot^M = (L_{I_{jp} \Box W_{jp}} S_{jp}^{M_{jp}}) \nonumber
\eeq
or in other terms:
\beq
M^{\op,\wedge} = (M_{jp}^{\op, \wedge}) \nonumber
\eeq
With this notion, we define $\Mhat = L_{I \Box W}\Sstdot^{\Mop}$ to be the model category of prestacks on $M$.

\subsection{Hyperdescent}
There is a slew of pseudo-representable objects we have to define before we get to the definition of hyperdescent. Those objects were initially defined in \cite{TV4} in the ungraded case, and we briefly reproduce their bi-graded definitions here for convenience's sake. \\

We say $F \in \Mhat$ is pseudo-representable if it is a small disjoint union of representable presheaves:
\beq
F \simeq \coprod_{u \in I}h_u = ( \coprod_{u_{ip} \in I_{ip}} h_{u_{ip}}) = ( F_{ip}) \nonumber
\eeq
that is it is entry-wise pseudo representable.\\

A pseudo-fibration is a morphism of pseudo-representable objects represented by a fibration in $M$. This breaks up into:
\begin{align}
	\oP_{l,p=0,1} \prod_{u_{lp} \in I_{lp}} \coprod_{v_{lp} \in J_{lp}} \Hom(h_{u_{lp}}, h_{v_{lp}}) \rarr \oP_{l,p=0,1} &\coprod_{v_{lp} \in J_{lp}} \Hom(h_{u_{lp}}, h_{v_{lp}}) \nonumber \\
	&\simeq \oP_{l,p} \coprod_{v_{lp} \in J_{lp}} \Hom_{M_{lp}}(u_{lp},v_{lp}) \nonumber
\end{align}
with entries that are pseudo-fibrations, so pseudo-fibrations are so entry-wise in $\uM$.\\

To avoid repetition, the reader is referred to \cite{TV} for the definition of pseudo-covering. It is clearly entry-wise a pseudo-covering.\\

For pseudo-representable hypercovers, let $x$ be a fibrant object of $M$. Then a pseudo-representable hypercover is an object $F_* \rarr h_x$ in $s\Mhat/h_x$ such that for any $n \geq 0$, the induced morphism:
\beq
F_n \rarr F^{\pDn} \times_{h_x^{\Dn}} h_x^{\Dn} \nonumber
\eeq
is a pseudo-fibration and pseudo-covering of pseudo-representable objects. We write $F_* = ( F_{*,ip})$ with $F_{*,ip} \in s\uMhat/h_{x_{ip}}$ with $x = (x_{ip})$ in $M$, so that:
\beq
\xymatrix{
	\oP_{i,p} F_{n,ip} \ar@{.>}[dr] \ar[r] &\Big( \oP_{i,p} F_{ip}^{\pDn}\Big) \times_{( \oP_{i,p} h_{x_{ip}}^{\pDn})} \Big( \oP_{i,p} h_{x_{ip}}^{\Dn} \Big) \ar@{=}[d] \\
&\oP_{i,p} F_{ip}^{\pDn} \times_{h_{x_{ip}}^{\pDn}} h_{x_{ip}}^{\Dn} } \nonumber
\eeq
showing that pseudo-representable hypercovers are defined entry-wise.\\

Note that $\Mhat = L_{I \Box W} \Sstdot^{\Mop}$ is naturally tensored and cotensored over $\Sstdot$ with external products and exponentials being defined objectwise. This makes $\Mhat$ into a bi-graded simplicial model category with bi-graded simplicial hom $\uHom$ and derived bi-graded simplicial hom that we will denote as in \cite{TV4} $\RomuHom$.\\

Regarding Yoneda, fix a cofibrant resolution functor $(\Gamma: M \rarr M^{\Delta}, \iota)$ with $\Gamma = (\Gamma_{ip})$ and $\iota = (\iota_{ip})$, that is for all $x \in M$, $\Gamma(x) = (\Gamma_{ip}(x_{ip}))$ is a co-simplicial object in $M$, cofibrant for the Reedy model structure on $M^{\Delta}$, together with a natural equivalence $\iota(x) : \Gamma(x) \rarr c^*(x)$, bi-graded constant co-simplicial object in $M$ at $x$. We have:
\begin{align}
	\uh= (\uh_{ip}): &M \rarr \sPrZtwosqu(M) \nonumber \\
	\uh_x = (\uh_{x_{ip}}): &\Mop \rarr \Sstdot \nonumber \\
	&y \mapsto \Hom_M(\Gamma(y),x)  \nonumber 
\end{align}
since $\uh_{ip}$ preserves fibrant objects and equivalences between them, so does $\uh$, so $\Ruh: \Ho(M) \rarr \Ho(\Mhat)$ is well-defined. Further:
\beq
\RomuHom_{\uMhat}(\uh_{x_{ip}}, F_{ip}) \cong F_{ip}(x_{ip}) \nonumber
\eeq
for $i,p=0,1$, leading to:
\beq
\RomuHom_{\Mhat}(\uh_x, F) \cong F(x) \nonumber
\eeq
This follows from the classical enriched Yoneda isomorphism $\RomuHom(h_x,F) \cong F(x)$, and the fact that $\Ho(h) \cong \bR \uh$ in $\Ho(\Mhat)$ entry-wise (\cite{TV}).\\

Finally, $F \in \Mhat$ is said to have hyperdescent if for any fibrant object $x \in M$, for any pseudo-representable hypercover $H_* \rarr h_x$, with realization $|H_*| = (|H_{*,ip}|)$, the induced morphism:
\beq
F(x) \simeq \RomuHom(h_x, F) \rarr \RomuHom(|\Hd|, F) \nonumber
\eeq
is an isomorphism in $\Ho(\Sstdot)$. A stack on $(M, \tet)$ is a prestack $F \in \Mhat$ that satisfies $\tet$-hyperdescent. We denote by $\Mtet$ the model category of stacks on $M$.
Another way to say this is by considering $H_{\beta}(x)$ for $x$ fibrant in $M$ the set of representatives of the set of isomorphism classes of objects $F_* \rarr h_x$ in $s\Mhat/h_x$ (see \cite{TV4} for details), those morphisms that are pseudo-representable hypercovers with a bound on the cardinality of each $F_n$. We have:
\beq
H_{\beta}(x) = (H_{\beta_{ip}}(x_{ip})) \nonumber
\eeq
Then we can equivalently say $\Mtet$ is the left Bousfield localization of $\Mhat$ with respect to morphisms in $H_{\beta}$.

\subsection{local equivalences}
Another way to define stacks is as follows. First consider the presheaf of connected components of $RF$:
\begin{align}
	\pizeropr(F) : & \Mop \rarr \Set_{* \cDot} \nonumber \\
	&x \mapsto \pi_0(RF(x)) = (\pi_0(R_{ip}F_{ip}(x_{ip}))) \nonumber
\end{align}
leading to:
\begin{align}
	\pizeroeq: & \Mhat \rarr \text{Pr}_{\bZ_2^2}(M) \nonumber \\
&F \mapsto \pizeropr(RF) \nonumber
\end{align}
factoring as:
\begin{align}
	\pizeroeq: & \Mhat \rarr \text{Pr}_{\bZ_2^2}(\Ho(M)) \nonumber \\
&F \mapsto \pizeroeq(F) \nonumber
\end{align}
which further factors as $\pizeroeq: \Ho(\Mhat) \rarr \Pr_{\bZ_2^2}(\Ho(M))$. We also have an evaluation functor:
\begin{align}
	j_x^*=( j_{x_{ip}}^*): &\Mhat \rarr (M/x)^{\wedge} \nonumber \\
	&F \mapsto F(x) = (F_{ip}(x_{ip})) \nonumber
\end{align}
with a left adjoint:
\beq
j_{x!} = ( j_{x_{ip}!}) :  (M/x)^{\wedge} = ( (\uM/x_{ip})^{\wedge} )\rarr \Mhat \nonumber
\eeq
so once right derived:
\beq
\Rjxstar: \Ho(\Mhat) \rarr \Ho((M/x)^{\wedge}) \nonumber
\eeq
Now for $F \in \Mhat$, $x \in M$ fibrant, $s \in \pizeroeq(F)(x)$ can be represented by a morphism $s: h_x \rarr F$ in $\Ho(\Mhat)$, giving by pullback $\Rjxstar s: \Rjxstar(h_x) \rarr \Rjxstar F$. Since we have a point $* \rarr \Rjxstar h_x$, by composition we get a global point $* \rarr \Rjxstar F$. All of this can be found in \cite{TV}. We now define, much as in this reference, for any $n>0$, the sheaf $\pi_n(F,s)$, by:
\beq
\pi_n(F,s) = \pi_0 \Big( \Rjxstar (F)^{\RDn} \times^h_{\Rjxstar(F)^{\RpDn}}* \Big) \nonumber
\eeq
this decomposes as follows:
\begin{align}
	\pi_n(F,s) &= \pi_0 \Bigg( \Big( \oP_{i,p} \Rjxipstar (F_{ip})^{\RDn}\Big) \times^h_{\oP_{i,p} \Rjxipstar(F_{ip})^{\RpDn}} \oP_{i,p} *_{ip} \Bigg) \nonumber \\
	&=\pi_0( \oP_{i,p} \Rjxipstar (F_{ip})^{\RDn} \times^h_{ \Rjxipstar(F_{ip})^{\RpDn}} *_{ip} ) \nonumber \\
	&= \oP_{i,p} \pi_0 (\Rjxipstar (F_{ip})^{\RDn} \times^h_{ \Rjxipstar(F_{ip})^{\RpDn}} *_{ip}) \nonumber \\
	&=\oP_{i,p} \pi_n(F_{ip},s_{ip}) \nonumber
\end{align}
where $s \in \pizeroeq(F)(x) = (\pizeroeq(F_{ip})(x_{ip}))$ so we write $s = (s_{ip})$.\\

We now define local equivalences exactly as in \cite{TV}. Those are also referred to as $\pi_*$-equivalences. A morphism $f: F \rarr G$ in $\Mhat$ is such an equivalence if the induced morphism $\pi_0(F) \rarr \pi_0(G)$ is an isomorphism on $\Ho(M)$, i.e. if it is entry-wise so, and additionally if for any fibrant $x$ in $M$, any $s \in \pizeroeq(F)(x)$, and $n > 0$, we have a bijection of sheaves on $\Ho(M/x)$: $\pi_n(F,s) \rarr \pi_n(G, f(s))$, and those are entry-wise isomorphisms as well. In other terms local equivalences are entry-wise local equivalences.\\

Finally given $M$, and for us this is really $\kDsAff_*$ we have in mind with the \'{e}tale topology, there is a closed model structure on $\sPrZtwosqu(M)$ called the entry-wise local projective model structure for which equivalences are local equivalences, cofibrations are cofibrations in $\Mhat$ for its projective model structure. It turns out that $\sPrZtwosqu(M)$ with this model structure is $\Mtet$. The proof of the existence of such a model structure is done entry-wise, and at that level this is proved in \cite{TV}. Now as pointed out above, we can also construct $\Mtet$ as the left Bousfield localization of $\Mhat$ with respect to $H = \{ |F_*| \rarr h_x \, | \, x \in M^f, F_* \in H_{\beta}(x) \}$. Maps in $H$ break down as $ |F_{*,ip}| \rarr h_{x_{ip}}$ so that it's clear we have $H = (H_{ip})$, hence:
\begin{align}
\Mtet &= L_H \Mhat \nonumber \\
	&= (L_{H_{ip}}M_{ip}^{\wedge}) \nonumber \\
	&=(M_{ip}^{\sim \, , \, \tet}) \nonumber 
\end{align}

\subsection{Truncations}
Define $X \in \Sstdot$ to be $n$-truncated if it is entry-wise $n$-truncated. $x \in \Ho(M_{*\cDot})$ is said to be $n$-truncated if it is so entry-wise. Then define $\pi_{\leq n}$-equivalences as in \cite{TV}. There is a model structure on $\sPrZtwosqu(M)$ called the $n$-truncated bi-graded local projective model structure for which equivalences are bi-graded $\pi_{\leq n}$-equivalences, cofibrations are cofibrations for the bi-graded projective model structure on $\Mhat$. That we have such a model structure follows directly for the same result in the ungraded case proved in \cite{TV}, since we can work entry-wise. This model category is denoted $M_{\leq n }^{\sim \, , \, \tet}$, and can be seen as a substitute for working with compactified theories in Theoretical Physics. We also have as a Corollary that this model category can be obtained as the left Bousfield localization of $\Mtet$ with respect to $\{  \partial \Delta^i \oT \uh_{x} \rarr  \Delta^i \oT \uh_x, i>n, x \in M \}$.

\newpage
\section{Supersymmetric stacks}
In what follows we will make use of the strictification theorem, which can be found in \cite{HS}, \cite{TV} and \cite{T}: for $C$ a category with a subcategory $S$ of morphisms, $M$ a cofibrantly generated model category, $M^{C,S}$ the localization of $M^C$ with respect to restricted diagrams, then:
\beq
L(M^{C,S}) \simeq \RuHom(L(C,S), LM) \nonumber
\eeq
where $L$ stands for Segal localization. In particular if $T$ is an $S$-site (\cite{TV}), then we just have $L(M^T) \cong \RuHom(T, LM)$. If $M = \Sstdot$, with $L(\Sstdot) = \Topstdot$, we have $L(\Sstdot^T) \cong \RuHom(T, \Topstdot)$. For us $T = L(\sksAlg_*)^{\op} = L(\kDsAff_*)$ will be denoted $\dksAff_*$. We define by:
\beq
\hatdksAff_* = \RuHom( \dksAff^{\op}_*, \Topstdot) \nonumber
\eeq
the Segal localization of $\sPrZtwosqu(\kDsAff_*)$. Then we define the category of stacks $\dksAfftet_*$ as a left exact localization of $\hatdksAff_*$, which can equivalently be obtained as the Segal localization of $\kDsAfftet_*$, itself the left Bousfield localization of $\sPrZtwosqu(\kDsAff_*)$, regarded as a simplicial category for strictication, and then as a bi-graded simplicial category.. Diagrammatically:
\beq
\xymatrix{
\kDsAff_*^{\wedge} \ar[d]_-{L_{\text{Bous}}} \ar[r]^-L & \RuHom(L(\kDsAff_*)^{\op}, \Topstdot ) = \hatdksAff_*  \ar[d]^{\text{LexLoc}} \nonumber\\
\kDsAfftet_* \ar[r]_{L} &\dksAfftet_* }
\eeq

The stacks we have defined are objects of $\kDsAfftet_*$, and after Segal localization, they become objects of the category $\dksAfftet_*$, where they become functors: $F:L(\sksAlg_*) \rarr \Topstdot$. Now we argue for every $A \in \sksAlg_*$, $\cX \in \Topstdot$ determined by constraint equations such as equations of motion, $\Psi:A \rarr \cX$ a map satisfying those equations, define $F(A) = \{\Psi(\sigma_i, \theta_i) \, | \, \sigma_i \in A_0, \theta_i \in A_1 , i = 0,1\}$ with the induced topology. Suppose $F$ thus defined is a stack. If in addition we have a notion of supersymmetric transformation on each $A \in \sksAlg_*$ under which $\Psi$ is well-behaved, then the resulting functor $F$ is called a supersymmetric stack. \\

In a first time we go over simple supersymmetry transformations using a standard reference such as \cite{GSW}. This will motivate our modifications, which will come afterwards.\\

\subsection{Ordinary Supersymmetry}
Everything in this section is from \cite{GSW}, and the reader is referred to that reference for more details. Suppose we have, for $M \in \sksMod$, two elements $\sigma^{\alpha}$ of $M_0 \in \skMod$, $\alpha = 1,2$, and two elements $\theta^A$ of $M_1$, $A = 1,2$. Those are seen as vector components. We also introduce matrices:
\beq
\rho^0 = \left(
           \begin{array}{cc}
             0 & -i \\
             i & 0 \\
           \end{array}
         \right) \nonumber
\eeq
and
\beq
\rho^1 = \left(
           \begin{array}{cc}
             0 & i \\
             i & 0 \\
           \end{array}
         \right) \nonumber
\eeq
We write $\bar{\theta} = \theta^T \rho^0$. Suppose elements of $M$ are variables. We can define the following derivatives:
\beq
Q_A = \frac{\partial }{\partial \bar{\theta}^A} + i(\rho^{\alpha} \theta)_A \partial_{\alpha} \nonumber
\eeq
defining what is called a supersymmetry transformation:
\begin{align}
\delta \theta^A &= [\bar{\epsilon}Q, \theta^A ] = \epsilon^A \nonumber \\
\delta \sigma^{\alpha} &=[\bar{\epsilon}Q, \sigma^{\alpha}] = i \bar{\eps} \rho^{\alpha} \theta \nonumber
\end{align}
where $\eps$ is an odd parameter. If one defines a field on $A$ valued in $\Top$, with $\mu$ being a dimensional index:
\beq
Y^{\mu}(\sigma, \theta) = X^{\mu}(\sigma) + \bar{\theta}\psi^{\mu}(\sigma) + \frac{1}{2} \bar{\theta}\theta B^{\mu}(\sigma)\nonumber
\eeq
then defining an ansatz for the supersymmetry transformation of $Y$ under $\delta$ as in:
\beq
\delta Y^{\mu}(\sigma, \theta) = \delta X^{\mu}(\sigma) + \bar{\theta}\delta \psi^{\mu}(\sigma) + \frac{1}{2} \bar{\theta}\theta \delta B^{\mu}(\sigma)\nonumber
\eeq
one finds:
\begin{align}
\delta X^{\mu} &= \bar{\eps}\psi^{\mu}\nonumber \\
\delta \psi^{\mu} &= -i \rho^{\alpha} \eps \partial_{\alpha} X^{\mu} + B^{\mu} \eps \nonumber \\
\delta B^{\mu} &= -i \bar{\eps}\rho^{\alpha} \partial_{\alpha} \psi^{\mu} \nonumber
\end{align}

\subsection{Supersymmetric maps}
For our purposes we need to have $Y$ valued in a bi-graded topological space. Note that in the definition of $Y^{\mu}$ in the preceding section, there is no odd part. Additionally, the supersymmetry transformation of $\theta$ produces $\eps$, another anticommuting variable, though if we want symmetry in our formalism, we would like to obtain an even element. Finally the supersymmetry transformations on $\sigma$ and $\theta$ are not symmetric themselves, so this is something we might want to enforce. Finally $\delta$ on each of $X^{\mu}$, $\psi^{\mu}$ and $B^{\mu}$ in the previous section is defined via an ansatz, so those transformations are really specific to each field and do not correspond to a single algebraic operation $\delta$, something we would like to have for a more functorial treatment of supersymmetry transformations.\\

We introduce $\cX = (\cX_{ip}) \in \Topstdot$, $\Psi: M \rarr \cX$, $M \in \ksMod_*$, with entries $\Psi_{ip}: M_{ip} \rarr \cX_{ip}$ with $i,p=0,1$. We have seen that $\ksMod_*$ has an internal Hom, its odd parts being made of morphisms in $\ksMod_*$ that change parity. Let $\delta$ be one such morphism.\\ 

We consider those $\delta$'s that provide supersymmetric transformations in the sense that they are elements of $\ucHom(M,M)_1$, and their definition displays some symmetry. For that purpose, we are led to introducing even parameters $\eps$ and $\lambda$, and matrices $\rho$ and $\gamma$, such that
\begin{align}
\delta_0 \sigma^{\alpha} &= \eps^T \rho^{\alpha} \theta \nonumber \\
\delta_1 \theta^A &=\lambda^T \gamma^A \sigma \nonumber
\end{align}
with $\alpha, A = 0,1$. Thus $\delta_0: M_0 \rarr M_1$ and $\delta_1: M_1 \rarr M_0$. The induced transformation on $\Psi$ is given by:
\begin{align}
\delta_0^* \Psi_1: &M_0 \rarr \cX_1 \nonumber \\
\delta_1^* \Psi_0: &M_1 \rarr \cX_0 \nonumber
\end{align}
or $\delta = \delta_0 \oP \delta_1 \in \ucHom(M,M)_1$. Thus:
\begin{align}
\delta^* \Psi(\sigma, \theta) &= \Psi(\delta(\sigma, \theta)) \nonumber \\
&= \Psi(\delta_0 \sigma \oP \delta_1 \theta) \nonumber \\
&= \Psi_0(\delta_1 \theta) \oP \Psi_1(\delta_0 \sigma)      ) \nonumber \\
&= \delta_1^*\Psi_0 \oP \delta_0^* \Psi_1 \nonumber
\end{align}
that is $\delta^*  = \delta_0^* \oP \delta_1^*$. We now generalize this construction to the case $M \in \sksMod_*$. This means we consider objects $\sigma_n$, $\theta_n$, $\gamma_n$, $\rho_n$, $\lambda_n$, $\eps_n$ and $\delta_n$. We have:
\beq
\xymatrix{
[n-1] \ar[d]_{d^i} \ar[r]^{M} & M_{n-1} \nonumber \\
[n] \ar[r]_{M} &M_n \ar[u]_{M(d^i) = d_i} }
\eeq
where $d_i \sigma_n^{\alpha} = (d_i \sigma^{\alpha})[n-1] = d^{i*}\sigma^{\alpha}[n-1] = \sigma^{\alpha} \circ d^i[n-1]$. $\delta$ is a natural transformation of simplicial objects, breaking into $\delta_0: M_{*0} \rarr M_{*1}$ and $\delta_1: M_{*1} \rarr M_{*0}$ with:
\beq
\xymatrix{
M_{n-1,0} \ar[r]^{\delta_{n-1,0}} & M_{n-1,1} \nonumber \\
M_{n,0} \ar[u]^{d_i} \ar[r]_{\delta_{n,0}} & M_{n,1} \ar[u]_{d_i}}
\eeq
as well as:
\beq
\xymatrix{
M_{n-1,0} \ar[d]_{s_j} \ar[r]^{\delta_{n-1,0}} & M_{n-1,1} \ar[d]_{s_j} \nonumber\\
M_{n,0}  \ar[r]_{\delta_{n,0}} & M_{n,1} }
\eeq
with $s_j$ the degeneracy map, defined by:
\beq
\xymatrix{
[n-1] \ar[r]^M &M_{n-1} \ar[d]^{M(s^j) = s_j}  \nonumber\\
[n] \ar[u]^{s^j} \ar[r]_M & M_n }
\eeq
where $d^i$ and $s^j$ are the usual connecting maps on $\Delta$ (\cite{GoJa}). We have similar commutative diagrams for $\delta_{*1}$. In all those, we have:
\begin{align}
\delta_{n,0}\sigma_n^{\alpha} &= \eps_n^T \rho_n^{\alpha} \theta_n \nonumber \\
\delta_{n,1} \theta^A_n &= \lambda^T_n \gamma^A_n \sigma_n \nonumber
\end{align}
with $\eps$, $\lambda$, $\rho$, $\gamma$ fixed, independently of the coordinates $\sigma$ and $\theta$, and:
\begin{align}
	\delta_n^* \Psi(\sigma_n, \theta_n) &=(\delta_{n,1}^* \Psi_0)(\theta_n) \oP (\delta_{n,0}^* \Psi_1)(\sigma_n) \nonumber \\
	&= \Psi_0(\delta_{n,1}\theta_n) \oP \Psi_1(\delta_{n,0} \sigma_n) \nonumber \\
	&= \Psi_0(\lambda_n^T \gamma_n^A \sigma_n) \oP \Psi_1(\eps_n^T \rho_n^{\alpha} \theta_n) \nonumber
\end{align}
the collection of which defines:
\beq
\delta^* F(A) = \{ \delta^* \Psi(\sigma, \theta) \; | \; \sigma \in A_0, \theta \in A_1 \} \nonumber
\eeq
thereby defining:
\begin{align}
\delta^*: \Topstdot &\rarr \Topstdot \nonumber \\
F(A) & \mapsto \delta^*F(A) \nonumber
\end{align}
objectwise presentation of:
\begin{align}
\delta^*: \widehat{\dksAff_*} & \rarr \widehat{\dksAff_*} \nonumber \\
F & \mapsto \delta^* F \nonumber
\end{align}
with $(\delta^*F)(A) = \delta^* F(A)$. If $\delta^*$ commutes with the left exact localization $\widehat{\dksAff_*} \rarr \dksAfftet_*$, then $F$ is referred to as a supersymmetric stack. In so doing, we preserve the physical origin of such a functor: to say that it is supersymmetric is in reference to the supersymmetry transformation in the base, that is the context is supersymmetric insofar as we have supersymmetri transformations, not that $F$ itself transforms in a symmetric manner under such transformations.

\end{document}